\documentclass[12pt]{amsart}

\newtheorem{theorem}{Theorem}[section]
\newtheorem{lemma}[theorem]{Lemma}
\newtheorem{corollary}[theorem]{Corollary}
\newtheorem{proposition}[theorem]{Proposition}

\newtheorem{conjecture}[theorem]{Conjecture}

\theoremstyle{plain}
\newtheorem{definition}[theorem]{Definition}

\newtheorem{remark}[theorem]{Remark}

\newtheorem{question}[theorem]{Question}

\theoremstyle{definition}

\theoremstyle{remark}

\numberwithin{equation}{section}

\usepackage{fancyhdr}
\usepackage{amscd}
\usepackage{amsmath}
\usepackage{amsthm}
\usepackage{amssymb}

\begin{document}

\title{Deformation of K-theoretic cycles}

\author{Sen Yang}
\address{Yau Mathematical Sciences Center, Tsinghua University
\\
Beijing, China}
\email{syang@math.tsinghua.edu.cn; senyangmath@gmail.com}

\date{}

\maketitle

\begin{abstract}
By using previous results in \cite{Y-2}, we answer the following two questions posed by Mark Green and Phillip Griffiths in chapter 10 of  \cite{GGtangentspace}(page 186-190): 
\begin{itemize}
\item  (1).Can one define $TZ^{p}(X)$(tangent space to cycle class group) in general ? 
 \item (2).Obstruction issues.
  \end{itemize} 
  
The highlight is the appearance of negative K-groups which detects the obstructions to deforming cycles.

\end{abstract}

\tableofcontents

\section{\textbf{Introduction}}
\label{Introduction}
For $X$ a smooth projective variety over a field $k$ of characteristic $0$, for each integer $p$ satisfying $1 \leqslant p \leqslant \mathrm{dim}(X)$, let $Z^{p}(X)$ denote the cycle class group,
\[
Z^{p}(X) = \bigoplus_{y \in X^{(p)}} \mathbb{Z} \cdot  \overline{\{ y\}}.
\]
The following question is posed by Green-Griffiths: 
\begin{question}[page 186 in \cite{GGtangentspace}] \label{question: question 1}
Can one define $TZ^{p}(X)$ in general ?
\end{question} 
Here, $TZ^{p}(X)$ is the tangent space to the cycle class group $Z^{p}(X)$. Since the abelian group $Z^{p}(X)$  is not a complex manifold or a scheme,  the known deformation theory, such as Kodaira-Spencer theory or the theory of Hilbert schemes, can't apply to this question directly. We consider $Z^{p}(-)$ as a functor and attempt to define the tangent space to this functor as usual
\[
 TZ^{p}(X):= \mathrm{Ker} \{Z^{p}(X \times \mathrm{Spec}(k[\varepsilon]/(\varepsilon^{2}))) \xrightarrow{\varepsilon=0}  Z^{p}(X) \},
\]
where $k[\varepsilon]/(\varepsilon^{2})$ is the ring of dual numbers.
Unfortunately, the classical definition of algebraic cycles can't  distinguish nilpotent, $Z^{p}(X \times \mathrm{Spec}(k[\varepsilon]/(\varepsilon^{2}))) =  Z^{p}(X)$, so this definition is clearly not the desirable one.

Green-Griffiths has answered this question for $p=1$(divisors) and $p=\mathrm{dim}(X)$(0-cycles) in \cite{GGtangentspace}. To give an example of what tangent spaces to cycle class groups are, we recall:
\begin{definition} [page 84-85 and page 141 in \cite{GGtangentspace}] \label{definition: GGdefinitionpointsurface}
For $X$ a smooth projective surface over a field $k$ of characteristic $0$, the tangent space $TZ^{2}(X)$ to the $0$-cycles on $X$ and the tangent subspace $TZ^{2}_{rat}(X)$ to the rational equivalence class are defined to be :
 \[
 TZ^{2}(X):=\bigoplus\limits_{x\in X^{(2)}} H_{x}^{2}(\Omega_{X/ \mathbb{Q}}^{1}), \\
 TZ^{2}_{rat}(X):=\mathrm{Im}(\partial_{1}^{1,-2}),
\]
where $\partial_{1}^{1,-2}$ is the differential of the Cousin complex of $\Omega_{X/ \mathbb{Q}}^{1}$,
\[
0 \to \Omega_{k(X)/ \mathbb{Q}}^{1}  \to \bigoplus\limits_{y \in X^{(1)}}H_{y}^{1}(\Omega_{X/\mathbb{Q}}^{1})  \xrightarrow{\partial^{1,-2}_{1}} \bigoplus\limits_{x\in X^{(2)}} H_{x}^{2}(\Omega_{X/ \mathbb{Q}}^{1}) \to 0.
\]
\end{definition}
It is worth noting that absolute differentials and local cohomology appear in this definition.

Moreover, Green-Griffiths points out that(page 186  in \cite{GGtangentspace}):
{ \it 
\begin{quote}
The technical issue that arises in trying straightforwardly extend the definitions given in the text for $p=n,1$ concerns cycles that are linear combinations of irreducible subvarieties
\[
 Z= \sum_{i}n_{i}Z_{i},
\]
where some $Z_{i}$ may not be the support of a locally Cohen-Macaulay scheme.
\end{quote}
}
To handle this technical issue, we look at generic points of $Z_{i}$s and need to use higher algebraic K-theory. In Section ~\ref{First order trivial deformation-tangent spaces}, we propose a definition of $TZ^{p}(X)$
in Definition ~\ref{definition: tangentspace cycles} for general $p$, generalizing Green-Griffiths' Definition \ref{definition: GGdefinitionpointsurface} above.

Considering an element $\tau \in TZ^{p}(X)$ as a first order deformation, Green-Griffiths asks whether we can successively deform $\tau$ to infinite order. It is well-known that the deformation of a subvariety $Y$, considered as an element of the Hilbert scheme $\mathrm{Hilb}(X)$, may be obstructed. However, Green-Griffiths predicts that we can eliminate obstructions, by considering $Y$ as an element of $Z^{p}(X)$:
\begin{conjecture} [page 187-190 in \cite{GGtangentspace}] \label{question: question 2}
$TZ^{p}(X)$  is formally unobstructed, see Conjecture \ref{conjecture: GGquestion2} in Section 3.2.
\end{conjecture}

We answer this conjecture in Theorem \ref{theorem: theorem obstruction}. The main idea for answering Question \ref{question: question 1} and Conjecture \ref{question: question 2} is to use Milnor K-theoretic cycles to replace the classical algebraic cycles.
In \cite{BalmerTensorChow}, Balmer defines K-theoretic Chow groups in terms of the derived category $D^{\mathrm{perf}}(X)$ obtained from the exact category of perfect complexes of $O_{X}$-modules. His idea is followed by Klein \cite{KleinChow} and the author \cite{Y-2}. 
By modifying Balmer's K-theoretic Chow groups \cite{BalmerTensorChow}, in \cite{Y-2}, we extend Soul\'e's variant of Bloch-Quillen identification from $X$ to its infinitesimally trivial deformations.
In this note, we continue using the techniques developed in \cite{Y-2} and focus on the geometry behind the formal definitions of K-theoretic cycles.

This note is organized as follows. 
We  recall Milnor K-theoretic cycles  and answer Green-Griffiths' Question ~\ref{question: question 1} in Section ~\ref{Definition of tangent spaces}, concrete examples of Milnor K-theoretic cycles from geometry(locally complete intersections) are also discussed. In Section ~\ref{Why take kernel} and Section ~\ref{Why use Milnor K-theory}, we explain two new aspects of  Milnor K-theoretic cycles, which are different from Balmer's \cite{BalmerTensorChow}, featuring negative K-groups and Milnor K-theory.

The relation between obstructions and negative K-groups is discussed in Section ~\ref{Obstructions and negative K-groups}.
We discuss obstruction issues and answer Green-Griffiths' Conjecture ~\ref{question: question 2} in Section ~\ref{Obstruction issues}.

\textbf{Notations and conventions}.

(1). K-theory used in this note will be Thomason-Trobaugh non-connective K-theory, if not stated otherwise. 

(2). For any abelian group $M$, $M_{\mathbb{Q}}$ denotes the image of $M$ in $M \otimes_{\mathbb{Z}} \mathbb{Q}$.

(3). $k[\varepsilon]/(\varepsilon^{2})$ is the ring of dual numbers.

\section{\textbf{First order deformation-tangent spaces}}
\label{First order trivial deformation-tangent spaces}
In this section, $X$ is a $d$-dimensional smooth projective variety over a field $k$ of characteristic $0$. For each positive integer $j$, $X_{j}:=X \times_{k} \mathrm{Spec}(k[t]/ t^{j+1})$ is the $j$-th order infinitesimally trivial deformation of $X$. In particular, we use $X[\varepsilon]$ to stand for $X_{1}$, i.e., $X[\varepsilon] = (X, O_{X}[t]/(t^2))$.

Recall that Milnor K-groups with support are rationally defined in terms of eigenspaces of Adams operations in \cite{Y-2}.
\begin{definition} [Definition 3.2 in \cite{Y-2}] \label{definition:Milnor K-theory with support}
Let $x_{j} \in X_{j}^{(i)}$, for any integer $m$, Milnor K-group with support $K_{m}^{M}(O_{X_{j},x_{j}} \ \mathrm{on} \ x_{j})$ is rationally defined to be 
\[
  K_{m}^{M}(O_{X_{j},x_{j}} \ \mathrm{on} \ x_{j}) := K_{m}^{(m+i)}(O_{X_{j},x_{j}} \ \mathrm{on} \ x_{j})_{\mathbb{Q}},
\] 
where $K_{m}^{(m+i)}$ is the eigenspace for $\psi^{k}=k^{m+i}$ and $\psi^{k}$ is the Adams operations.

\end{definition}

For each positive integer $p$, there exists the following variant of Gersten complex, see Theorem 3.14 in \cite{Y-2},
{\footnotesize
\begin{align*}
 0  \to  & \bigoplus_{x_{j} \in X_{j}^{(0)}}K_{p}^{M}(O_{X_{j},x_{j}}) \to \dots \to \bigoplus_{x_{j} \in X_{j}^{(p-1)}}K_{1}^{M}(O_{X_{j},x_{j}} \ \mathrm{on} \ x_{j}) \\
   & \xrightarrow{d_{1,X_{j}}^{p-1,-p}} \bigoplus_{x_{j} \in X_{j}^{(p)}}K_{0}^{M}(O_{X_{j},x_{j}} \ \mathrm{on} \ x_{j})  \xrightarrow{d_{1,X_{j}}^{p,-p}} \bigoplus_{x_{j} \in X_{j}^{(p+1)}}K_{-1}^{M}(O_{X_{j},x_{j}} \ \mathrm{on} \ x_{j})   \to \cdots\\ 
   & \to \bigoplus_{x_{j} \in X_{j}^{(d)}}K_{q-d}^{M}(O_{X_{j},x_{j}} \ \mathrm{on} \ x_{j}) \to 0.
\end{align*}
}

\begin{definition} [Definition 3.4 and Definition 3.15 in \cite{Y-2}] \label{definition: Milnor K-theoretic Chow groups}
For each positive integer $p$, the $p$-th Milnor K-theoretic cycles and Milnor K-theoretic rational equivalence of $X_{j}$, denoted $Z^{M}_{p}(D^{\mathrm{Perf}}(X_{j}))$ and $Z^{M}_{p,rat}(D^{\mathrm{Perf}}(X_{j}))$, are defined as
\[
Z^{M}_{p}(D^{\mathrm{Perf}}(X_{j})) := \mathrm{Ker}(d_{1,X_{j}}^{p,-p}),
\]
\[
Z^{M}_{p,rat}(D^{\mathrm{Perf}}(X_{j})) := \mathrm{Im}(d_{1,X_{j}}^{p-1,-p}).
\]

The $p$-th Milnor K-theoretic Chow group of $X_{j}$ is defined to be:
\[
  CH^{M}_{p}(D^{\mathrm{perf}}(X_{j})) := \dfrac{\mathrm{Ker}(d_{1,X_{j}}^{p,-p})}{\mathrm{Im}(d_{1,X_{j}}^{p-1,-p})}.
\]
\end{definition}

The reasons why we take the kernel of $d^{p,-p}_{1, X_{j}}$ to define $Z^{M}_{p}(D^{\mathrm{perf}}(X_{j}))$ and why we use Milnor K-groups with support, i.e., certain eigenspaces of Thomason-Trobaugh K-groups,  not the entire Thomason-Trobaugh K-groups, are explained in Section ~\ref{Why take kernel} and Section ~\ref{Why use Milnor K-theory} respectively.
\subsection{Definition of tangent spaces}
\label{Definition of tangent spaces}
For $Y \subset X$ a subvariety of codimension $p$,
let $i: Y \to X$ be the inclusion, then $i_{\ast}O_{Y}$ is a coherent $O_{X}$-module and  can be resolved by a bounded complex of vector bundles on $X$.  Let $Y^{'}$ be a first order deformation of $Y$, that is, $Y^{'} \subset X[\varepsilon] $ such that $Y^{'}$ is flat over $\mathrm{Spec}(k[\varepsilon])$ and $Y^{'} \otimes_{k[\varepsilon]} k \cong Y$. Then $i_{\ast}O_{Y'}$ can be resolved by a bounded complex of vector bundles on $X[\varepsilon]$, where $i: Y^{'} \to X[\varepsilon]$.

 Let $D^{\mathrm{perf}}(X[\varepsilon])$ denote the derived category of perfect complexex of $O_{X}[\varepsilon]$-modules, and let $\mathcal{L}_{(i)}(X[\varepsilon]) \subset D^{\mathrm{perf}}(X[\varepsilon])$ be defined as
\[
  \mathcal{L}_{(i)}(X[\varepsilon]) := \{ E \in D^{\mathrm{perf}}(X[\varepsilon]) \mid \mathrm{codim_{Krull}(supph(E))} \geq -i \},
\]
where the closed subset $\mathrm{supph(E)} \subset X$ is the support of the total homology of the perfect complex $E$.
The resolution of $i_{\ast}O_{Y'}$, which is a perfect complex of $O_{X}[\varepsilon]$-module supported on $Y$, defines an element of the Verdier quotient $\mathcal{L}_{(-p)}(X[\varepsilon])/\mathcal{L}_{-(p-1)}(X[\varepsilon])$, denoted $[i_{\ast}O_{Y'}]$.

If $Y \subset X$ is a locally complete intersection of codimension $p$, there exists an open affine $U$( $\subset X$) such that $U \cap Y$ is defined by a regular sequence $(f_{1}, \cdots, f_{p})$, where $f_{i} \in O_{X}(U)$. 
Locally on $U$, $Y^{'}$ is given by lifting $f_{1}, \cdots, f_{p}$ to 
 $f_{1}+ \varepsilon g_{1}, \cdots, f_{p}+ \varepsilon g_{p}$, where $g_{i} \in O_{X}(U)$. 
 
 We use $F_{\bullet}(f_{1}+\varepsilon g_{1}, \cdots, f_{p}+\varepsilon g_{p})$ to denote the Koszul complex associated to the  regular sequence $f_{1}+\varepsilon g_{1}, \cdots, f_{p}+\varepsilon g_{p}$,  which is a resolution of $O_{X}(U)[\varepsilon]/(f_{1}+\varepsilon g_{1}, \cdots, f_{p}+\varepsilon g_{p})$:
\[
 \begin{CD}
  0 @>>> F_{p} @>A_{p}>> F_{p-1} @>A_{p-1}>>  \dots @>A_{2}>> F_{1} @>A_{1}>> F_{0},
 \end{CD}
\]
where each $F_{i}=\bigwedge^{i} (O_{X}(U)[\varepsilon])^{\bigoplus p}$ and $A_{i}: \bigwedge^{i} (O_{X}(U)[\varepsilon])^{\bigoplus p}  \to \bigwedge^{i-1} (O_{X}(U)[\varepsilon])^{\bigoplus p}$ are defined as usual.
And one can define tangent to this Koszul complex, which is given by the following commutative diagram(we assume $g_{2}= \cdots = g_{p}=0$ for simplicity):
\begin{equation}
\begin{cases}
 \begin{CD}
   F_{\bullet}(f_{1}, f_{2},\cdots, f_{p}) @>>> O_{X}(U)/(f_{1}, f_{2}, \cdots,  f_{p}) \\
  F_{p}(\cong O_{X}(U)) @>  g_{1}df_{2} \wedge \cdots \wedge df_{p}>> F_{0} \otimes \Omega_{O_{X}(U)/ \mathbb{Q}}^{p-1}(\cong\Omega_{O_{X}(U)/ \mathbb{Q}}^{p-1}), 
 \end{CD}
 \end{cases}
\end{equation}
where $d=d_{\mathbb{Q}}$. 

However, in general, $Y \subset X$ may not be a locally complete intersection and the length of the perfect complex $[i_{\ast}O_{Y'}]$, which is the resolution of $i_{\ast}O_{Y'}$, may not equal to $p$. To modify this, instead of considering $[i_{\ast}O_{Y'}]$ which is an element of  the Verdier quotient $\mathcal{L}_{(-p)}(X[\varepsilon])/\mathcal{L}_{(-p-1)}(X[\varepsilon])$ , we consider its image in the idempotent completion $(\mathcal{L}_{(-p)}(X[\varepsilon])/\mathcal{L}_{(-p-1)}(X[\varepsilon]))^{\#}$,  denoted $[i_{\ast}O_{Y'}]^{\#}$. And we have the following result: 
\begin{theorem} \cite{B-3} \label{theorem: Balmer theorem}
 For each $i \in \mathbb{Z}$, localization induces an equivalence
\begin{equation}
 (\mathcal{L}_{(i)}(X[\varepsilon])/\mathcal{L}_{(i-1)}(X[\varepsilon]))^{\#}  \simeq \bigsqcup_{x[\varepsilon] \in X[\varepsilon]^{(-i)}}D_{{x[\varepsilon]}}^{\mathrm{perf}}(X[\varepsilon])
\end{equation}
between the idempotent completion of the Verdier quotient $\mathcal{L}_{(i)}(X[\varepsilon])/\mathcal{L}_{(i-1)}(X[\varepsilon])$ and the coproduct over $x[\varepsilon] \in X[\varepsilon]^{(-i)}$ of the derived category of perfect complexes of $ O_{X[\varepsilon],x[\varepsilon]}$-modules with homology supported on the closed point $x[\varepsilon] \in \mathrm{Spec}(O_{X[\varepsilon],x[\varepsilon]})$. And consequently, one has 
\[
 K_{0}((\mathcal{L}_{(i)}(X[\varepsilon])/\mathcal{L}_{(i-1)}(X[\varepsilon]))^{\#})  \simeq \bigoplus_{x[\varepsilon] \in X[\varepsilon]^{(-i)}}K_{0}(D_{{x[\varepsilon]}}^{\mathrm{perf}}(X[\varepsilon])).
\]
\end{theorem}

Let $y$ be the generic point of $Y$, $Y$ is generically defined by a regular sequence of length $p$: $f_{1}, \cdots, f_{p}$, where $f_{1}, \cdots, f_{p} \in O_{X, y}$.  $Y'$ is generically given by lifting  $f_{1}, \cdots, f_{p}$ to $f_{1}+\varepsilon g_{1}, \cdots, f_{p}+\varepsilon g_{p}$, where $g_{1}, \cdots, g_{p} \in O_{X, y}$.
We use $F_{\bullet}(f_{1}+\varepsilon g_{1}, \cdots, f_{p}+\varepsilon g_{p})$ to denote the Koszul complex associated to the  regular sequence $f_{1}+\varepsilon g_{1}, \cdots, f_{p}+\varepsilon g_{p}$,  which is a resolution of $O_{X,y}[\varepsilon]/(f_{1}+\varepsilon g_{1}, \cdots, f_{p}+\varepsilon g_{p})$.

Under the equivalence (2.2), the localization at the generic point $y$ sends $[i_{\ast}O_{Y'}]^{\#}$ to the Koszul complex $F_{\bullet}(f_{1}+\varepsilon g_{1}, \cdots, f_{p}+\varepsilon g_{p})$:
\[
[i_{\ast}O_{Y'}]^{\#} \to F_{\bullet}(f_{1}+\varepsilon g_{1}, \cdots, f_{p}+\varepsilon g_{p}).
\]
And one can define tangent to the Koszul  complex $F_{\bullet}(f_{1}+\varepsilon g_{1}, \cdots, f_{p}+\varepsilon g_{p})$ similarly as (2.1), which defines an element of $H^{p}_{y}(\Omega^{p-1}_{X/ \mathbb{Q}})$.

\begin{remark} \label{remark: notallKoszul}
In general, we don't know whether the above kind of Koszul complexes can generate the Grothendieck group $\bigoplus\limits_{x \in X^{(p)}}K^{M}_{0}(O_{X,x}[\varepsilon] \ \mathrm{on} \ x)$ or not. So we can't use only these Koszul complexes to define tangent space to cycle class groups and have to use the following formal approach.
\end{remark}

We recall that the  Milnor K-theoretic cycles and Chow groups in Definition ~\ref{definition: Milnor K-theoretic Chow groups} recover the classical ones for $X$: 
\begin{theorem} [Theorem 3.16 in \cite{Y-2}]  \label{theorem: recover}
For $X$ a smooth projective variety over a field k of characteristic 0, for each positive integer $p$, let 
$Z^{p}(X)$, $Z^{p}_{rat}(X)$ and $CH^{p}(X)$ denote the classical
$p$-cycles, rational equivalence and Chow groups respectively, then we have the following identifications
\[
Z^{M}_{p}(D^{\mathrm{perf}}(X)) =  Z^{p}(X)_{\mathbb{Q}},
\]
\[
Z^{M}_{p,rat}(D^{\mathrm{perf}}(X)) =  Z^{p}_{rat}(X)_{\mathbb{Q}},
\]
\[
CH^{M}_{p}(D^{\mathrm{perf}}(X)) = CH^{p}(X)_{\mathbb{Q}}.
\]
\end{theorem}

Recall that the tangent space to a functor $\mathcal{F}$, denoted $T\mathcal{F}(X)$, is defined to be 
\[
  T\mathcal{F}(X) :=\mathrm{Ker} \{ \mathcal{F}(X[\varepsilon]) \xrightarrow{\varepsilon =0 } \mathcal{F}(X)\}.
\]
Considering $Z^{M}_{p}(D^{\mathrm{perf}}(-))$ as a functor, we are guided to the following definition, which answers Green-Griffiths' {Question ~\ref{question: question 1}:
\begin{definition} \label{definition: tangentspace cycles}
For $X$ a smooth projective variety over a field k of characteristic 0, for each positive integer $p$, the tangent space to $p$-cycles, denoted $TZ^{p}(X)$, is defined to be 
\[
  TZ^{p}(X) := TZ^{M}_{p}(D^{\mathrm{perf}}(X) =\mathrm{Ker} \{ Z^{M}_{p}(D^{\mathrm{perf}}(X[\varepsilon])) \xrightarrow{\varepsilon =0 } Z^{M}_{p}(D^{\mathrm{perf}}(X)\}.
\]
Similarly, the tangent space to rational equivalent classes, denoted $TZ^{p}_{rat}(X)$, is defined to be 
\[
  TZ^{p}_{rat}(X) :=TZ^{M}_{p,rat}(D^{\mathrm{perf}}(X) = \mathrm{Ker} \{ Z^{M}_{p,rat}(D^{\mathrm{perf}}(X[\varepsilon])) \xrightarrow{\varepsilon =0 } Z^{M}_{p,rat}(D^{\mathrm{perf}}(X)\}.
 \]
\end{definition}

\newpage

The following theorem has been proved in \cite{DHY, Y-2}.
\begin{theorem} [\cite{DHY}, Theorem 3.14 in \cite{Y-2}] \label{theorem: TanMInorFinal}
Let $X$ be a smooth projective variety over a field k of characteristic 0. For each integer $p\geqslant 1$, there exists the following commutative diagram in which the Zariski sheafification of each column is a flasque resolution of $ \Omega_{X/ \mathbb{Q}}^{p-1}$,  $K^{M}_{p}(O_{X[\varepsilon]})$ and $K^{M}_{p}(O_{X})$
respectively. The left arrows are induced by Chern character from K-theory to negative cyclic homology and the right ones are the natural maps sending $\varepsilon$ to $0$:
{\scriptsize
\[
  \begin{CD}
     0 @. 0 @. 0\\
      @VVV @VVV @VVV\\
     \Omega_{k(X)/ \mathbb{Q}}^{p-1} @<<< K^{M}_{p}(k(X)[\varepsilon]) @>>> K^{M}_{p}(k(X))  \\
     @VVV @VVV @VVV\\
      \bigoplus\limits_{x \in X^{(1)}}H_{x}^{1}(\Omega_{X/\mathbb{Q}}^{p-1}) @<<< \bigoplus\limits_{x[\varepsilon] \in X[\varepsilon] ^{(1)}}K^{M}_{p-1}(O_{X,x}[\varepsilon] \ \mathrm{on} \ x[\varepsilon]) @>>>  \bigoplus\limits_{x \in X ^{(1)}}K^{M}_{p-1}(O_{X,x} \ \mathrm{on} \ x) \\
     @VVV @VVV @VVV\\
      \dots @<<< \dots @>>> \dots \\ 
      @VVV @VVV @VVV\\
     \bigoplus\limits_{x \in X^{(p-1)}}H_{x}^{p-1}(\Omega_{X/\mathbb{Q}}^{p-1}) @<<< \bigoplus\limits_{x[\varepsilon] \in X[\varepsilon] ^{(p-1)}}K^{M}_{1}(O_{X,x}[\varepsilon] \ \mathrm{on} \ x[\varepsilon])
      @>>> \bigoplus\limits_{x \in X^{(p-1)}}K^{M}_{1}(O_{X,x} \ \mathrm{on} \ x) \\
     @V\partial_{1}^{p-1,-p}VV @Vd_{1,X[\varepsilon]}^{p-1,-p}VV @Vd_{1,X}^{p-1,-p}VV\\
     \bigoplus\limits_{x \in X^{(p)}}H_{x}^{p}(\Omega_{X/\mathbb{Q}}^{p-1}) @<<< \bigoplus\limits_{x[\varepsilon] \in X[\varepsilon] ^{(p)}}K^{M}_{0}(O_{X,x}[\varepsilon] \ \mathrm{on} \ x[\varepsilon])
      @>>> \bigoplus\limits_{x \in X^{(p)}}K^{M}_{0}(O_{X,x} \ \mathrm{on} \ x) \\
     @V\partial_{1}^{p,-p}VV @Vd_{1,X[\varepsilon]}^{p,-p}VV @Vd_{1,X}^{p,-p}VV\\
     \bigoplus\limits_{x \in X^{(p+1)}}H_{x}^{p+1}(\Omega_{X/\mathbb{Q}}^{p-1}) @<<< \bigoplus\limits_{x[\varepsilon] \in X[\varepsilon] ^{(p+1)}}K^{M}_{-1}(O_{X,x}[\varepsilon] \ \mathrm{on} \ x[\varepsilon])
      @>>> \bigoplus\limits_{x \in X^{(p+1)}}K^{M}_{-1}(O_{X,x} \ \mathrm{on} \ x)=0 \\
     @VVV @VVV @VVV\\
     \cdots @<<< \cdots @>>> \cdots \\ 
     @VVV @VVV @VVV\\
     \bigoplus\limits_{x \in X^{(d)}}H_{x}^{d}(\Omega_{X/\mathbb{Q}}^{p-1}) @<<< \bigoplus\limits_{x[\varepsilon] \in X[\varepsilon] ^{(d)}}K^{M}_{q-d}(O_{X,x}[\varepsilon] \ \mathrm{on} \ x[\varepsilon]) @>>> \bigoplus\limits_{x \in X^{(d)}}K^{M}_{q-d}(O_{X,x} \ \mathrm{on} \ x) \\
     @VVV @VVV @VVV\\
      0 @. 0 @. 0.
  \end{CD}
\]
}

\end{theorem}

This diagram enables us to compute $TZ^{p}(X)$ and $TZ^{p}_{rat}(X)$. A quick diagram chasing shows
\begin{theorem} \label{theorem: compute tangent space}
Let $X$ be a smooth projective variety over a field k of characteristic 0. For each integer $p\geqslant 1$, we have the following identifications:
\[
  TZ^{p}(X) \cong \mathrm{Ker}(\partial_{1}^{p,-p}),
\]
\[
  TZ^{p}_{rat}(X) \cong \mathrm{Im}(\partial_{1}^{p-1,-p}).
 \]
\end{theorem}

Evidently, $TZ^{p}_{rat}(X)$ is a subspace of $TZ^{p}(X)$. We use the quotient space to define the tangent space to Chow groups:
\begin{definition} \label{definition: tangentspace Chowgroups}
Let $X$ be a smooth projective variety over a field k of characteristic 0. For each integer $p\geqslant 1$, the tangent space to $CH^{p}(X)$,  denoted $TCH^{p}(X)$, is defined to be 
\[
  TCH^{p}(X) := \dfrac{TZ^{p}(X)}{ TZ^{p}_{rat}(X) }.
\]
\end{definition}

\begin{theorem} \label{theorem: agreeBlochformal}
$TCH^{p}(X)$ agrees with the formal tangent space $T_{f}CH^{p}(X)$ defined by Bloch \cite{BlochLectures}, where $T_{f}CH^{p}(X)=H^{p}(X, \Omega_{X/\mathbb{Q}}^{p-1})$.
\end{theorem}

\begin{proof}
It immediately follows from the fact that the Zariski sheafification of the left column in Theorem ~\ref{theorem: TanMInorFinal} is a flasque resolution of $\Omega_{X/\mathbb{Q}}^{p-1}$.

\end{proof}

For $X$ a smooth projective surface over a field $k$ of characteristic $0$, by taking $p=2$ in Theorem ~\ref{theorem: compute tangent space}, we immediately see that 
\begin{corollary} \label{corollary: agreetangentspace}
For $X$ a smooth projective surface over a field $k$ of characteristic $0$, Green and Griffiths' definitions of $TZ^{2}(X)$ and  $TZ^{2}_{rat}(X)$, recalled in Definition \ref{definition: GGdefinitionpointsurface}, agree with the formal Definition \ ~\ref{definition: tangentspace cycles}.
\end{corollary}

Next, we provide concrete examples of  Milnor K-theoretic cycles which are from geometry.
 Let $Y \subset X$ be a locally complete intersection of codimension $p$. For a point $x \in Y \subset X$, there exists an open affine $U$( $\subset X$) containing $x$ such that $U \cap Y$ is defined by a regular sequence $f_{1}, \cdots, f_{p}$, where $f_{i} \in O_{X,x}$.  
 Let $Y^{'}$ be a first order deformation of $Y$, locally on $U$, $Y^{'}$ is given by lifting $f_{1}, \cdots, f_{p}$ to 
 $f_{1}+ \varepsilon g_{1}, \cdots, f_{p}+ \varepsilon g_{p}$, where $g_{i} \in O_{X,x}$. 
 
Let $y$ be the generic point of $Y$, then $O_{X,y}=(O_{X,x})_{(f_{1}, \cdots, f_{p})}$ and we see $Y$  is generically defined by $f_{1}, \cdots, f_{p}$. We use $F_{\bullet}(f_{1}+\varepsilon g_{1}, \cdots, f_{p}+\varepsilon g_{p})$ to denote the Koszul complex associated to the  regular sequence $f_{1}+\varepsilon g_{1}, \cdots, f_{p}+\varepsilon g_{p}$,  which is a resolution of $O_{X, y}[\varepsilon]/(f_{1}+\varepsilon g_{1}, \cdots, f_{p}+\varepsilon g_{p})$:
\[
 \begin{CD}
  0 @>>> F_{p} @>A_{p}>> F_{p-1} @>A_{p-1}>>  \dots @>A_{2}>> F_{1} @>A_{1}>> F_{0},
 \end{CD}
\]
where each $F_{i}=\bigwedge^{i} (O_{X,y}[\varepsilon])^{\bigoplus p}$ and $A_{i}: \bigwedge^{i} (O_{X,y}[\varepsilon])^{\bigoplus p}  \to \bigwedge^{i-1} (O_{X,y}[\varepsilon])^{\bigoplus p}$ are defined as usual.
Then $F_{\bullet}(f_{1}+\varepsilon g_{1}, \cdots, f_{p}+\varepsilon g_{p}) \in K_{0}(O_{X,y}[\varepsilon] \ \mathrm{on} \ y[\varepsilon])$.

\begin{theorem} [ Prop 4.12  of \cite{GilletSoule} ]  \label{theorem: GilletSoule}
The Adams operations $\psi^{k}$ defined on perfect complexes, defined by Gillet-Soul\'e in \cite{GilletSoule}, satisfy
$\psi^{k}(F_{\bullet}(f_{1}+\varepsilon g_{1}, \cdots, f_{p}+\varepsilon g_{p})) = k^{p}F_{\bullet}(f_{1}+\varepsilon g_{1}, \cdots, f_{p}+\varepsilon g_{p})$.
\end{theorem}

Hence, $F_{\bullet}(f_{1}+\varepsilon g_{1}, \cdots, f_{p}+\varepsilon g_{p})$ is of eigenweight $p$ and can be considered as an element of $K^{(p)}_{0}(O_{X,y}[\varepsilon] \ \mathrm{on} \ y[\varepsilon])_{\mathbb{Q}}$: 
\[
F_{\bullet}(f_{1}+\varepsilon g_{1}, \cdots, f_{p}+\varepsilon g_{p}) \in K^{(p)}_{0}(O_{X,y}[\varepsilon] \ \mathrm{on} \ y[\varepsilon])_{\mathbb{Q}}=K^{M}_{0}(O_{X,y}[\varepsilon] \ \mathrm{on} \ y[\varepsilon]).
\]

Moreover, we shall show $F_{\bullet}(f_{1}+\varepsilon g_{1}, \cdots, f_{p}+\varepsilon g_{p})$ lies in the kernel of 
\[
d^{p,-p}_{1,X[\varepsilon]}:  \bigoplus\limits_{x[\varepsilon] \in X[\varepsilon]^{(p)}} K^{M}_{0}(O_{X,x}[\varepsilon] \ \mathrm{on} \ x[\varepsilon]) \to    \bigoplus\limits_{x[\varepsilon] \in X[\varepsilon]^{(p+1)}} K^{M}_{-1}(O_{X,x}[\varepsilon] \ \mathrm{on} \ x[\varepsilon])
\]
so that it is a Milnor K-theoretic $p$-cycle:
\begin{theorem}  \label{theorem: Inkernel-LCI}
For $X$ a smooth projective variety over a field $k$ of characteristic $0$,
let $Y \subset X$ be a locally complete intersection of codimension $p$. 
Suppose $Y$ is locally defined by a regular sequence $f_{1}, \cdots, f_{p}$, where $x$ is a point on $Y$ and $f_{i} \in O_{X,x}$. A first order deformation $Y^{'}$ is locally given by lifting $f_{1}, \cdots, f_{p}$ to 
 $f_{1}+ \varepsilon g_{1}, \cdots, f_{p}+ \varepsilon g_{p}$, where $g_{i} \in O_{X,x}$.  Then the Koszul complex $ F_{\bullet}(f_{1}+\varepsilon g_{1}, \cdots, f_{p}+\varepsilon g_{p}) \in \mathrm{Ker}(d^{p,-p}_{1,X[\varepsilon]})$, i.e., $F_{\bullet}(f_{1}+\varepsilon g_{1}, \cdots, f_{p}+\varepsilon g_{p}) \in Z^{M}_{p}(D^{\mathrm{perf}}(X[\varepsilon]))$.
\end{theorem}

The strategy for proving this theorem is to  use the map induced by Chern character from K-theory to negative cyclic homology
\[
\mathrm{Ch}:  K^{M}_{0}(O_{X,y}[\varepsilon] \ \mathrm{on} \ y[\varepsilon]) \to H_{y}^{p}(\Omega_{X/\mathbb{Q}}^{p-1}), 
\]
mapping $F_{\bullet}(f_{1}+\varepsilon g_{1}, \cdots, f_{p}+\varepsilon g_{p})$ to an element of  $H_{y}^{p}(\Omega_{X/\mathbb{Q}}^{p-1})$, and then show its image under the differential 
\[
\partial_{1}^{p,-p}: \bigoplus\limits_{x \in X^{(p)}}H_{x}^{p}(\Omega_{X/\mathbb{Q}}^{p-1}) \to   \bigoplus\limits_{x \in X^{(p+1)}}H_{x}^{p+1}(\Omega_{X/\mathbb{Q}}^{p-1}) 
\]
is zero.

\begin{proof}
 The map(left arrows) induced by Chern character from K-theory to negative cyclic homology in the commutative diagram of Theorem \ref{theorem: TanMInorFinal}
\[
\mathrm{Ch}:  K^{M}_{0}(O_{X,y}[\varepsilon] \ \mathrm{on} \ y[\varepsilon]) \to H_{y}^{p}(\Omega_{X/\mathbb{Q}}^{p-1}),
\]
can be described by a beautiful construction of Ang\'enoil and Lejeune-Jalabert, see Lemme  3.1.1 on page 24 and Definition 3.4 on page 29  in \cite{A-LJ} for details or Section 3 of  \cite{Y-3} for a brief summary.

For our purpose, the Ch map on the Koszul complex $ F_{\bullet}(f_{1}+\varepsilon g_{1}, \cdots, f_{p}+\varepsilon g_{p})$
can be described easily. For simplicity, we assume $g_{2}= \cdots = g_{p}=0$ in the following.  To the Koszul complex, 
\[
 \begin{CD}
  0 @>>> F_{p} @>A_{p}>> F_{p-1} @>A_{p-1}>>  \dots @>A_{2}>> F_{1} @>A_{1}>> F_{0},
 \end{CD}
\]
one defines the following class
\[
 \frac{1}{p!}dA_{1}\circ dA_{2}\circ \dots \circ dA_{p}, 
\]
where $d=d_{\mathbb{Q}}$ and each $dA_{i}$ is the matrix of absolute differentials. In other words,
\[
 dA_{i} \in \mathrm{Hom}(F_{i},F_{i-1}\otimes \Omega_{O_{X,y}[\varepsilon]/\mathbb{Q}}^{1}).
\]
The truncation map $\rfloor  \dfrac{\partial}{\partial \varepsilon}\mid_{\varepsilon =0}$ sends $\frac{1}{p!}dA_{1}\circ dA_{2}\circ \dots \circ dA_{p}$ to $g_{1}df_{2}\wedge \cdots \wedge df_{p}$.
So the image of  $ F_{\bullet}(f_{1}+\varepsilon g_{1}, \cdots, f_{p})$, under the Ch map, in  $H_{y}^{p}(\Omega_{X/\mathbb{Q}}^{p-1})$ is represented by the following diagram(an element of $Ext^{p}(O_{X,y}/(f_{1}, f_{2}, \cdots,  f_{p}), \Omega_{O_{X,y}/ \mathbb{Q}}^{p-1})$), 
\begin{equation}
\begin{cases}
 \begin{CD}
   F_{\bullet}(f_{1}, f_{2},\cdots, f_{p}) @>>> O_{X,y}/(f_{1}, f_{2}, \cdots,  f_{p}) \\
  F_{p}(\cong O_{X,y}) @>  g_{1}df_{2} \wedge \cdots \wedge df_{p}>> F_{0} \otimes \Omega_{O_{X,y}/ \mathbb{Q}}^{p-1}(\cong\Omega_{O_{X,y}/ \mathbb{Q}}^{p-1}).
 \end{CD}
 \end{cases}
\end{equation}

The regular sequence $f_{1}, \cdots, f_{p}$, where $f_{i} \in O_{X,x}$, can be extended to be a system of parameter $f_{1}, \cdots, f_{p}, f_{p+1}, \cdots, f_{d}$ in $O_{X,x}$. The prime ideals $Q_{i}:=(f_{1}, \cdots, f_{p}, f_{i})$, where $i= p+1, \cdots, d$, define generic points $z_{i} \in X^{(p+1)}$. In the following, we consider the prime $Q_{p+1}$ and the generic point $z_{p+1}$, other cases work similarly.  

Let $P=(f_{1}, \cdots, f_{p})$ be the prime ideal defining the generic point (of $Y$)$y \in X^{(p)}$, $O_{X,y}=(O_{X,z_{p+1}})_{P}$. The above diagram(2.3) can be rewritten as, denoted $[\alpha]$,
\begin{equation}
\begin{cases}
 \begin{CD}
   F_{\bullet}(f_{1}, f_{2},\cdots, f_{p}) @>>> (O_{X,z_{p+1}})_{P}/(f_{1}, f_{2}, \cdots,  f_{p}) \\
  F_{p}(\cong (O_{X,z_{p+1}})_{P}) @>  \dfrac{g_{1} f_{p+1}}{f_{p+1}}df_{2} \wedge \cdots \wedge df_{p}>> F_{0} \otimes \Omega_{(O_{X,z_{p+1}})_{P} / \mathbb{Q}}^{p-1}(\cong\Omega_{(O_{X,z_{p+1}})_{P}/ \mathbb{Q}}^{p-1}).
 \end{CD}
 \end{cases}
\end{equation}
Here, $F_{\bullet}(f_{1}, f_{2},\cdots, f_{p})$ is of the form
\[
 \begin{CD}
  0 @>>> F_{p} @>A_{p}>> F_{p-1} @>A_{p-1}>>  \dots @>A_{2}>> F_{1} @>A_{1}>> F_{0},
 \end{CD}
\]
where each $F_{i}=\bigwedge^{i}((O_{X,z_{p+1}})_{P})^{\oplus p}$.
Since $f_{p+1} \notin (f_{1}, \cdots, f_{p})$, $f_{p+1}^{-1}$ exists in $(O_{X,z_{p+1}})_{P}$, we can write $g_{1}df_{2} \wedge \cdots \wedge df_{p}= \dfrac{g_{1} f_{p+1}}{f_{p+1}}df_{2} \wedge \cdots \wedge df_{p}$.

The image of the diagram(2.4) under the differential 
\[
\partial_{1}^{p,-p}: \bigoplus\limits_{x \in X^{(p)}}H_{x}^{p}(\Omega_{X/\mathbb{Q}}^{p-1}) \to   \bigoplus\limits_{x \in X^{(p+1)}}H_{x}^{p+1}(\Omega_{X/\mathbb{Q}}^{p-1}) 
\]
 is represented by the following diagram
\[
\begin{cases}
 \begin{CD}
   F_{\bullet}(f_{1}, f_{2},\cdots, f_{p}, f_{p+1}) @>>> O_{X, z_{p+1}}/(f_{1}, f_{2}, \cdots,  f_{p}, f_{p+1}) \\
  F_{p+1}(\cong O_{X,z_{p+1}}) @>  g_{1} f_{p+1}df_{2} \wedge \cdots \wedge df_{p}>> F_{0} \otimes \Omega_{O_{X,z_{p+1}}/ \mathbb{Q}}^{p-1}(\cong\Omega_{O_{X,z_{p+1}}/ \mathbb{Q}}^{p-1}).
 \end{CD}
 \end{cases}
\]
The complex $F_{\bullet}(f_{1}, f_{2},\cdots, f_{p}, f_{p+1})$ is of the form
\[
 \begin{CD}
  0 @>>>  \bigwedge^{p+1}(O_{X, z_{p+1}})^{\oplus p+1} @>A_{p+1}>> \bigwedge^{p}(O_{X, z_{p+1}})^{\oplus p+1} @>>>  \cdots.
 \end{CD}
\]
 Let $\{e_{1}, \cdots, e_{p+1} \}$ be a basis of $(O_{X, z_{p+1}})^{\oplus p+1} $, the map $A_{p+1}$ is 
\[
 e_{1}\wedge \cdots \wedge e_{p+1}  \to \sum^{p+1}_{j=1}(-1)^{j}f_{j}e_{1}\wedge \cdots \wedge \hat{e_{j}} \wedge \cdots e_{p+1},
\]
where $\hat{e_{j}}$ means to omit the $j^{th}$ term. 

Noting $f_{p+1}$ appears in $A_{p+1}$, 
\[
g_{1} f_{p+1}df_{2} \wedge \cdots \wedge df_{p} \equiv 0 \in Ext_{O_{X,z_{p+1}}}^{p+1}(O_{X, z_{p+1}}/(f_{1}, f_{2}, \cdots,  f_{p}, f_{p+1}), \Omega_{O_{X,z_{p+1}}/ \mathbb{Q}}^{p-1}),
\]
 so $\partial_{1}^{p,-p}([\alpha])=0$. There exists the following commutative diagram, 
which is part of the commutative diagram in Theorem \ref{theorem: TanMInorFinal}
\[
\begin{CD}
 \bigoplus\limits_{x \in X^{(p)}}H_{x}^{p}(\Omega_{X/\mathbb{Q}}^{p-1}) @<<< \bigoplus\limits_{x[\varepsilon] \in X[\varepsilon]^{(p)}}K^{M}_{0}(O_{X,x}[\varepsilon] \ \mathrm{on} \ x[\varepsilon]) \\
     @V\partial_{1}^{p,-p}VV @Vd_{1,X[\varepsilon]}^{p,-p}VV \\
  \bigoplus\limits_{x \in X^{(p+1)}}H_{x}^{p+1}(\Omega_{X/\mathbb{Q}}^{p-1}) @< \cong << \bigoplus\limits_{x[\varepsilon] \in X[\varepsilon] ^{(p+1)}}K^{M}_{-1}(O_{X,x}[\varepsilon] \ \mathrm{on} \ x[\varepsilon]).
      \end{CD}
      \]
This gives the following commutative diagram
\[
\begin{CD}
 [\alpha ]      @<\mathrm{Ch}<<  F_{\bullet}(f_{1}+ \varepsilon g_{1}, f_{2},\cdots, f_{p}) \\
     @V\partial_{1}^{p,-p}VV @Vd_{1,X[\varepsilon]}^{p,-p}VV \\
     \partial_{1}^{p,-p}([\alpha ]) =0   @<\mathrm{Ch}<\cong <   d_{1,X[\varepsilon]}^{p,-p}(F_{\bullet}(f_{1}+ \varepsilon g_{1}, f_{2},\cdots, f_{p})),
      \end{CD}
      \]
which shows $d_{1,X[\varepsilon]}^{p,-p}(F_{\bullet}(f_{1}+ \varepsilon g_{1}, f_{2},\cdots, f_{p})) = 0$.

\end{proof}
 
 In general, $Y \subset X$ may not be a locally complete intersection, and the associated Koszul complex  $F_{\bullet}(f_{1}+ \varepsilon g_{1}, f_{2}, \cdots, f_{p})$ may not be a Milnor K-theoretic $p$-cycle. However, we can find another subscheme $Z \subset X$ of codimension $p$ and $Z' \in \mathrm{T_{Z}Hilb}^{p}(X)$ such that the two Koszul complexes associated $Y'$ and $Z'$ defines a Milnor K-theoretic $p$-cycle

To fix notations, let $W \subset Y$ be a subvariety of codimension 1 in $Y $, with generic point $w$. One assumes $W$ is generically defined by $f_{1}, f_{2}, \cdots, f_{p}, f_{p+1}$ and $Y$ is generically defined by $f_{1}, f_{2}, \cdots, f_{p}$. Let $y$ be the generic point of $Y$, one  has $O_{X,y}=(O_{X,w})_{P}$, where $P$ is the idea $(f_{1}, f_{2}, \cdots, f_{p}) \subset O_{X,w}$. 

$Y'$ is generically given by $(f_{1}+ \varepsilon g_{1}, f_{2}+\varepsilon g_{2}, \cdots, f_{p}+ \varepsilon g_{p})$, where $g_{i} \in O_{X,y}=(O_{X,w})_{P}$. For simplicity, we assume $g_{2}= \cdots = g_{p}=0$. We can write $g_{1}= \dfrac{a}{b}$, where $a,b \in O_{X,w}$ and $b \notin P$. $b$ is either in or not in the maximal idea $(f_{1}, f_{2}, \cdots, f_{p}, f_{p+1}) \subset O_{X,w}$.

 \begin{theorem} [Theorem 4.7 in \cite{Y-3}]  \label{theorem: MainThmCycles}
 For $Y' \in \mathrm{T_{Y}Hilb}^{p}(X)$ which is generically defined by $(f_{1}+ \varepsilon g_{1}, f_{2}, \cdots, f_{p})$, where $g_{1}= \dfrac{a}{b} \in O_{X,y}=(O_{X,w})_{P}$, we use $F_{\bullet}(f_{1}+ \varepsilon g_{1}, f_{2}, \cdots, f_{p})$ to denote the Koszul complex associated to $f_{1}+ \varepsilon g_{1}, f_{2}, \cdots, f_{p}$,
 
\begin{itemize}
\item Case 1: If $b \notin (f_{1}, f_{2}, \cdots, f_{p}, f_{p+1})$,  then \
 $F_{\bullet}(f_{1}+ \varepsilon g_{1}, f_{2}, \cdots, f_{p}) \in Z^{M}_{p}(D^{\mathrm{Perf}}(X[\varepsilon]))$. \ \

\item Case 2: if $b \in (f_{1}, f_{2}, \cdots, f_{p}, f_{p+1})$, we reduce to considering $b=f_{p+1}$. Then there exists $Z \subset X$ which is generically defined by $(f_{p+1}, f_{2}, \cdots, f_{p})$ and exists $Z' \in \mathrm{T_{Z}Hilb}^{p}(X)$ which is generically defined by $(f_{p+1}+ \varepsilon \dfrac{a}{f_{1}}, f_{2}, \cdots, f_{p})$
such that $F_{\bullet}(f_{1}+ \varepsilon \dfrac{a}{f_{p+1}}, f_{2}, \cdots, f_{p})+F_{\bullet}(f_{p+1}+ \varepsilon \dfrac{a}{f_{1}}, f_{2}, \cdots, f_{p}) \in Z^{M}_{p}(D^{\mathrm{Perf}}(X[\varepsilon]))$. 
 
\end{itemize}
 \end{theorem}

\subsection{Why take kernel}
\label{Why take kernel} 
 In this subsection, we explain the reasons why we use the kernel of $d^{p,-p}_{1, X_{j}}$ to define $Z^{M}_{p}(D^{\mathrm{perf}}(X_{j}))$ in Definition ~\ref{definition: Milnor K-theoretic Chow groups}, instead of taking $\bigoplus\limits_{x_{j} \in X_{j}^{(p)}}K^{M}_{0}(O_{X_{j},x_{j}} \ \mathrm{on} \ x_{j})$. 
 
 1). As explained in the beginning of Section 2.1, in general, the length of the perfect complex $[i_{\ast}O_{Y'}]$, which is the resolution of $i_{\ast}O_{Y'}$, may not equal to $p$. To modify this, we need to look at its image $[i_{\ast}O_{Y'}]^{\#}$ in the idempotent completion $(\mathcal{L}_{(-p)}(X[\varepsilon])/\mathcal{L}_{(-p-1)}(X[\varepsilon]))^{\#}$. From the K-theoretic viewpoint, taking idempotent completion can result in the appearance of negative K-groups. We should include negative K-groups into 
 the study of deformation of cycles, so we use the kernel of $d^{p,-p}_{1, X_{j}}$ to define $Z^{M}_{p}(D^{\mathrm{perf}}(X_{j}))$. In Section 3.2, negative K-groups will be used for obstruction issues.
 
 2). From the geometric viewpoint,  taking the kernel of 
 \[
 d^{p,-p}_{1, X[\varepsilon]}:  \bigoplus\limits_{x[\varepsilon] \in X[\varepsilon]^{(p)}}K^{M}_{0}(O_{X,x}[\varepsilon] \ \mathrm{on} \ x[\varepsilon])  \to  \bigoplus\limits_{x[\varepsilon] \in X[\varepsilon] ^{(p+1)}}K^{M}_{-1}(O_{X,x}[\varepsilon] \ \mathrm{on} \ x[\varepsilon])
  \]
 to define $Z^{M}_{p}(D^{\mathrm{perf}}(X[\varepsilon]))$ can produce the desirable tangent space. This can be explained by the following example.
 
Let $X$ be a smooth projective surface over a field $k$ of characteristic $0$, we consider the 1-cycles $Z^{1}(X)$ on $X$ and study its tangent space $TZ^{1}(X)$. For simplicity, we look at the sheaf level, that is, we look at the tangent sheaf  $\underline{T}Z^{1}(X)$ to the 1-cycles $Z^{1}(X)$.

Let $\underline{Z}^{1}(X)$ be the Zariski sheaf of 1-cycles on $X$, we have the following short exact sequence of sheaves:
\[
0 \to O_{X}^{\ast} \to K(X)^{\ast}  \xrightarrow{\mathrm{div}} \underline{Z}^{1}(X) \to 0,
\]
where $K(X)$ is the function field of $X$. It is known that the tangent sheaves to $O_{X}^{\ast}$ and $K(X)^{\ast}$ are 
$O_{X}$ and $K(X)$ respectively, and there exists the following short exact sequence of sheaves:
\begin{equation}
0 \to O_{X} \to K(X)  \to PP_{X} \to 0,
\end{equation}
where $PP_{X}$ is the sheaf of principal parts. This suggests that
\begin{definition}[page 100 \cite{GGtangentspace}] \label{definition: GG's PP part}
The tangent sheaf  $\underline{T}Z^{1}(X)$ to the 1-cycles $Z^{1}(X)$ is defined to be
\[
\underline{T}Z^{1}(X):=PP_{X}.
\]

\end{definition} 

To related this definition with the formal Definition \ref{definition: tangentspace cycles}, we note that the Cousin resolution of  $O_{X}$ is 
\begin{equation}
0 \to O_{X}  \to K(X) \to \bigoplus\limits_{y \in X^{(1)}}i_{y, \ast}H_{y}^{1}(O_{X})  \xrightarrow{\partial^{1,-2}_{1}} \bigoplus\limits_{x\in X^{(2)}} i_{x, \ast}H_{x}^{2}(O_{X}) \to 0.
\end{equation}

For $X$ a smooth projective surface over a field $k$ of characteristic $0$, taking $p=1$ in Theorem \ref{theorem: compute tangent space},
we see the tangent sheaf  is $\mathrm{ker}(\partial^{1,-2}_{1})$.  The two exact sequences(2.5)(2.6) show that 
\[
PP_{X} \cong K(X)/O_{X} \cong \mathrm{ker}(\partial^{1,-2}_{1}).
\]
This proves:
\begin{corollary} \label{corollary: agreeGGSheaf}
For 1-cycles $Z^{1}(X)$ on $X$, the formal Definition \ref{definition: tangentspace cycles}(at the sheaf level) agrees with the Definition \ref{definition: GG's PP part} by Green-Griffiths.

\end{corollary}

If we don't use the kernel of $d^{1,-1}_{1, X[\varepsilon]}$, but use $\bigoplus\limits_{y[\varepsilon] \in X[\varepsilon]^{(1)}}K_{0}(O_{X,y}[\varepsilon] \ \mathrm{on} \ y[\varepsilon])$, to define Milnor K-theoretic 1-cycles, then the tangent sheaf  becomes $\bigoplus\limits_{y \in X^{(1)}} H_{y}^{1}(O_{X})$, which is obviously not the desirable one.

In the next, combining with Green-Griffiths' results in \cite{GGtangentspace}, we construct a concrete element of the kernel of $d^{p,-p}_{1, X[\varepsilon]}$.

\begin{theorem} \label{theorem: whykernel}
Let $X$ be a smooth projective surface over a field $k$ of characteristic $0$, for $p=1$ in $\mathrm{Theorem \ ~\ref{theorem: TanMInorFinal}}$, we have the following commutative diagram. The left arrows are induced by Chern character from K-theory to negative cyclic homology and the right ones are the natural maps sending $\varepsilon$ to $0$:
\[\displaystyle
  \begin{CD}
     0 @. 0 @. 0\\
     @VVV @VVV @VVV\\
     k(X) @<<<  K_{1}(k(X)[\varepsilon])_{\mathbb{Q}} @>>> K_{1}(k(X))_{\mathbb{Q}} \\
     @VVV @VVV @VVV\\
     \bigoplus\limits_{y \in X^{(1)}}H_{y}^{1}(O_{X}) @<<<  \bigoplus\limits_{y[\varepsilon] \in X[\varepsilon]^{(1)}}K_{0}(O_{X,y}[\varepsilon] \ \mathrm{on} \ y[\varepsilon])_{\mathbb{Q}} @>>> \bigoplus\limits_{y\in X^{(1)}}K_{0}(O_{X,y} \ \mathrm{on} \ y)_{\mathbb{Q}} \\
     @V \partial^{1,-1}_{1}VV @Vd^{1,-1}_{1,X[\varepsilon]}VV @VVV\\
     \bigoplus\limits_{x\in X^{(2)}}H_{x}^{2}(O_{X}) @<<\cong<  \bigoplus\limits_{x[\varepsilon] \in X[\varepsilon]^{(2)}}K_{-1}(O_{X,x}[\varepsilon] \ \mathrm{on} \ x[\varepsilon])_{\mathbb{Q}} @>>> \bigoplus\limits_{x\in X^{(2)}}K_{-1}(O_{X,x} \ \mathrm{on} \ x)_{\mathbb{Q}} =0 \\
     @VVV @VVV @VVV\\
      0 @. 0 @. 0.
  \end{CD}
\]
\end{theorem}

Let's explain why one can use $K_{0}(O_{X,y} \ \mathrm{on} \ y)_{\mathbb{Q}}$ to replace $K^{M}_{0}(O_{X,y}\ \mathrm{on} \ y)$(defined in Definition ~\ref{definition:Milnor K-theory with support}) in the above diagram. 
One notes that $K^{(j)}_{0}(O_{X,y} \ \mathrm{on} \ y)_{\mathbb{Q}}\cong K^{(j-1)}_{0}(k(y))=0$, except for $j=1$. That is ,
\[
 K^{(1)}_{0}(O_{X,y}\ \mathrm{on} \ y)_{\mathbb{Q}} = K_{0}(O_{X,y} \ \mathrm{on} \ y)_{\mathbb{Q}}.
\]
This says $K^{M}_{0}(O_{X,y}\ \mathrm{on} \ y) = K_{0}(O_{X,y} \ \mathrm{on} \ y)_{\mathbb{Q}}$.
Similar arguments for other K-groups in the middle and right columns in the above diagram.

Let $Y_{1}$ and $Y_{2}$ be two curves on $X$ with generic point $y_{1}$ and $y_{2}$ respectively.  For simplicity, we work locally in Zariski topology and assume $Y_{1}$ and $Y_{2}$ intersect transversely at a point $x$. Around the point $x$, we can write 
\[
  Y_{1}= \mathrm{div}(f_{1});  \ Y_{2}= \mathrm{div}(f_{2}).
\]

Take $g \in O_{X,x}$ such that $g(x) \neq 0$,  we consider $O_{X,x}[\varepsilon]/ (f_{1}f_{2} + \varepsilon g)$. The Koszul resolution of  $O_{X,x}[\varepsilon]/ (f_{1}f_{2} + \varepsilon g)$,
\[
L^{\bullet}: 0 \to O_{X,x}[\varepsilon] \xrightarrow{f_{1}f_{2} + \varepsilon g} O_{X,x}[\varepsilon],
\]
defines an element of $K_{0}(\mathcal{L}_{-1}(X[\varepsilon])/\mathcal{L}_{-2}(X[\varepsilon]))^{\#})$.

\begin{theorem} \label{theorem: GG-inKernel}
$L^{\bullet} \in \mathrm{Ker}(d^{1,-1}_{1,X[\varepsilon]})$, i.e., $L^{\bullet} \in Z^{M}_{1}(D^{\mathrm{perf}}(X[\varepsilon]))$.
\end{theorem}

\begin{proof}
Under the isomorphism in Theorem \ref{theorem: Balmer theorem}
\[
 K_{0}((\mathcal{L}_{(-1)}(X[\varepsilon])/\mathcal{L}_{(-2)}(X[\varepsilon]))^{\#})  \simeq \bigoplus\limits_{y[\varepsilon] \in X[\varepsilon]^{(1)}}K_{0}(D_{{y[\varepsilon]}}^{\mathrm{perf}}(X[\varepsilon])),
\]
$L^{\bullet}$ decomposes into the direct sum of 
\[
 L^{\bullet}_{1}: 0 \to (O_{X,x})_{(f_{1})}[\varepsilon] \xrightarrow{f_{1} + \varepsilon \dfrac{g}{f_{2}}}  (O_{X,x})_{(f_{1})}[\varepsilon] 
\]
and
\[
 L^{\bullet}_{2}: 0 \to (O_{X,x})_{(f_{2})}[\varepsilon] \xrightarrow{f_{2} + \varepsilon \dfrac{g}{f_{1}}}  (O_{X,x})_{(f_{2})}[\varepsilon].
\]
Noting $O_{X,y_{1}}=(O_{X,x})_{(f_{1})}$, we have $L^{\bullet}_{1} \in K_{0}(O_{X,y_{1}}[\varepsilon] \ \mathrm{on} \ y_{1}[\varepsilon])$. Similarly,  $L^{\bullet}_{2} \in K_{0}(O_{X,y_{2}}[\varepsilon] \ \mathrm{on} \ y_{2}[\varepsilon])$.

The following diagram, associated to $L^{\bullet}_{1}$,
\[
\begin{cases}
 \begin{CD}
   (O_{X,x})_{(f_{1})} @>f_{1}>>  (O_{X,x})_{(f_{1})} @>>> (O_{X,x})_{(f_{1})}/(f_{1}) @>>> 0  \\
   (O_{X,x})_{(f_{1})} @>\frac{g}{f_{2}} >> (O_{X,x})_{(f_{1})},
 \end{CD}
\end{cases}
\]
gives an element $\alpha$ in $Ext_{O_{X,y_{1}}}^{1}(O_{X,y_{1}}/(f_{1}), O_{X,y_{1}})$, which further defines an element in $H_{y_{1}}^{1}(O_{X})$ and it is the image of $L^{\bullet}_{1}$
under the map in Theorem ~\ref{theorem: whykernel}: 
\begin{equation}
\mathrm{Ch}: \bigoplus\limits_{y[\varepsilon] \in X[\varepsilon]^{(1)}}K_{0}(O_{X,y}[\varepsilon] \ \mathrm{on} \ y[\varepsilon])_{\mathbb{Q}} \to \bigoplus\limits_{y \in X^{(1)}}H_{y}^{1}(O_{X}).
\end{equation}

Similarly, the following diagram, associated to $L^{\bullet}_{2}$,
\[
\begin{cases}
 \begin{CD}
   (O_{X,x})_{(f_{2})} @>f_{2}>>  (O_{X,x})_{(f_{2})} @>>> (O_{X,x})_{(f_{2})}/(f_{2}) @>>> 0  \\
   (O_{X,x})_{(f_{2})} @>\frac{g}{f_{1}} >> (O_{X,x})_{(f_{2})},
 \end{CD}
\end{cases}
\]
gives an element $\beta$ in $Ext_{O_{X,y_{2}}}^{1}(O_{X,y_{2}}/(f_{2}), O_{X,y_{2}})$, which further defines an element in $H_{y_{2}}^{1}(O_{X})$ and it is the image of $L^{\bullet}_{2}$
under the Ch map. 

One notes $\partial^{1,-1}_{1}$ maps $\alpha$  in $H^{2}_{x}(O_{X})$ to :
\[
\begin{cases}
 \begin{CD}
   O_{X,x} @>(f_{2}, -f_{1})^{T}>>  O_{X,x}^{\oplus 2} @>(f_{1}, f_{2})>> O_{X,x} @>>> O_{X,x}/(f_{1}, f_{2}) @>>> 0  \\
   O_{X,x} @>g >> O_{X,x}.
 \end{CD}
\end{cases}
\]

Similarly, $\partial^{1,-1}_{1}$ maps $\beta$   in $H^{2}_{x}(O_{X})$ to :
\[
\begin{cases}
 \begin{CD}
   O_{X,x} @>(f_{1}, -f_{2})^{T}>>  O_{X,x}^{\oplus 2} @>(f_{2}, f_{1})>> O_{X,x} @>>> O_{X,x}/(f_{1}, f_{2}) @>>> 0  \\
   O_{X,x} @>g >> O_{X,x}.
 \end{CD}
\end{cases}
\]

Noting the commutative diagram below 
\[\displaystyle
  \begin{CD}
    O_{X,x} @>(f_{2}, -f_{1})^{T}>>  O_{X,x}^{\oplus 2} @>(f_{1}, f_{2})>> O_{X,x} @>>> O_{X,x}/(f_{1}, f_{2}) @>>> 0  \\
   @V-1VV @VMVV @V1VV @V=VV\\
  O_{X,x} @>(f_{1}, -f_{2})^{T}>>  O_{X,x}^{\oplus 2} @>(f_{2}, f_{1})>> O_{X,x} @>>> O_{X,x}/(f_{2}, f_{1}) @>>> 0,
  \end{CD}
\]
where $M$ stands for the matrix 
\[
\begin{pmatrix} 0 & 1 \\ 1 & 0 \end{pmatrix},
\]
Green-Griffiths observes that  $\partial^{1,-1}_{1}(\alpha)$ and $\partial^{1,-1}_{1}(\beta)$ are negative of each other in $Ext_{O_{X,x}}^{2}(O_{X,x}/(f_{1},f_{2}), O_{X,x})$. Hence, $\partial^{1,-2}_{1}(\alpha + \beta)$ is $0$ in $H^{2}_{x}(O_{X})$. Therefore, $d^{1,-1}_{1,X[\varepsilon]}(L^{\bullet}) = 0$ because of the commutative diagram:

\[\displaystyle
  \begin{CD}
\bigoplus\limits_{y \in X^{(1)}}H_{y}^{1}(O_{X}) @<<<  \bigoplus\limits_{y[\varepsilon] \in X[\varepsilon]^{(1)}}K_{0}(O_{X,y}[\varepsilon] \ \mathrm{on} \ y[\varepsilon]) \\
     @V \partial^{1,-1}_{1}VV @Vd^{1,-1}_{1,X[\varepsilon]}VV \\
     \bigoplus\limits_{x\in X^{(2)}} H_{x}^{2}(O_{X}) @<<\cong<  \bigoplus\limits_{x[\varepsilon] \in X[\varepsilon]^{(2)}}K_{-1}(O_{X,x}[\varepsilon] \ \mathrm{on} \ x[\varepsilon]).
     \end{CD}
\]

\end{proof}


The above argument seems formal, so it's convenient to intuitively explain the meaning of taking the kernel of $d^{1,-1}_{1,X[\varepsilon]}$. This has been done by using residue by Green-Griffiths \cite{GGtangentspace}.

\textbf{Alternative explanation by using residue, due to Green-Griffiths \cite{GGtangentspace}}(page 103-104 and the summary on page 119)
To fix notations, let $Y_{1}$ and $Y_{2}$ be two curves on $X$.
It is well-known that tangent vectors to the curves $Y_{1}$ and $Y_{2}$ are given by normal vector fields, 
\[
 v_{1} \in H^{0}(N_{Y_{1}/X}),  v_{2} \in H^{0}(N_{Y_{2}/X}).
\]

For simplicity, we work locally in Zariski topology and assume $Y_{1}$ and $Y_{2}$ intersect transversely at a point $x$. Around the point $x$, we can write 
\[
  Y_{1}= \mathrm{div}(f_{1});  \ Y_{2}= \mathrm{div}(f_{2}).
\]
Then $v_{1}$ and $v_{2}$ can be expressed as
\[
 v_{1} = w_{1}\dfrac{\partial}{\partial f_{1}}, v_{2} = w_{2}\dfrac{\partial}{\partial f_{2}},
\]
for some functions $w_{1}$ and $w_{2}$. For our purpose, we take $w_{1}= \dfrac{g}{f_{2}}$ and $w_{2}= \dfrac{h}{f_{1}}$, then 
\[
 v_{1} = \dfrac{g}{f_{2}} \dfrac{\partial}{\partial f_{1}}, v_{2} = \dfrac{h}{f_{1}} \dfrac{\partial}{\partial f_{2}}.
\]

For $\omega = d f_{1} \wedge d f_{2}$, we consider the Poincar\'e residue:
\begin{equation}
\begin{cases}
 \begin{CD}
  v_{1} \rfloor \omega =  \mathrm{Res}_{Y_{1}}(\dfrac{g d f_{1} \wedge d f_{2}}{f_{1}f_{2}}) = \dfrac{g d f_{2}}{f_{2}} \in \Omega^{1}_{K(Y_{1})/ \mathbb{C}}; \\
v_{2} \rfloor \omega =  \mathrm{Res}_{Y_{2}}(\dfrac{h d f_{1} \wedge d f_{2}}{f_{1}f_{2}}) = - \dfrac{h d f_{1}}{f_{1}} \in \Omega^{1}_{K(Y_{2})/ \mathbb{C}}.
 \end{CD}
\end{cases}
\end{equation} 
We further take the residue at $x$:
\[
\mathrm{Res}_{x}(\dfrac{g d f_{2}}{f_{2}}) = g, \mathrm{Res}_{x}(- \dfrac{h d f_{1}}{f_{1}}) = -h.
\]
The sum of the residues is
\[
\mathrm{Res}_{x}(\dfrac{g d f_{2}}{f_{2}}) + \mathrm{Res}_{x}(- \dfrac{h d f_{1}}{f_{1}}) = g-h.
\]
When $g=h$, the sum of the residues  is $0$.

\textbf{Conclusion}: for $v_{1} = \dfrac{g}{f_{2}} \dfrac{\partial}{\partial f_{1}}$ and $v_{2} = \dfrac{g}{f_{1}} \dfrac{\partial}{\partial f_{2}}$, 
\[
\mathrm{Res}_{x}(v_{1} \rfloor \omega ) + \mathrm{Res}_{x}(v_{1} \rfloor \omega) =0.
\]

How does this connect to K-groups?

For normal vectors
\[
 v_{1} = \dfrac{g}{f_{2}} \dfrac{\partial}{\partial f_{1}}, v_{2} = \dfrac{g}{f_{1}} \dfrac{\partial}{\partial f_{2}}, 
\]
$v_{1}$ corresponds  to $f_{1} + \varepsilon \dfrac{g}{f_{2}}$ and $v_{2}$ corresponds  to $f_{2} + \varepsilon \dfrac{g}{f_{1}}$. In other words, $v_{1}$ corresponds  to the complex
\[
 L^{\bullet}_{1}: 0 \to (O_{X,x})_{(f_{1})}[\varepsilon] \xrightarrow{f_{1} + \varepsilon \dfrac{g}{f_{2}}}  (O_{X,x})_{(f_{1})}[\varepsilon] 
\]
and $v_{2}$ corresponds  to the complex
\[
 L^{\bullet}_{2}: 0 \to (O_{X,x})_{(f_{2})}[\varepsilon] \xrightarrow{f_{2} + \varepsilon \dfrac{g}{f_{1}}}  (O_{X,x})_{(f_{2})}[\varepsilon].
\]

\textbf{Conclusion}: $\mathrm{Res}_{x}(v_{1} \rfloor \omega ) + \mathrm{Res}_{x}(v_{2} \rfloor \omega) =0$  corresponds to $(L^{\bullet}_{1} + L^{\bullet}_{2}) \in \mathrm{Ker}(d^{1,-1}_{1,X_{1}})$ in Theorem \ref{theorem: GG-inKernel}.

\begin{remark}
One may ask why there is no necessary to take kernel in Quillen's or Soul\'e 's proofs of Bloch's formula in \cite{Quillen, Soule}.  That's because negative K-groups are zero in this case, $K_{-1}(k(x))=0$.  If we take kernel,  the cycles class group $Z^{p}(X)$ is still identified with $\bigoplus\limits_{x \in X^{(p)}}K_{0}(k(x))$.
\end{remark}

\subsection{Why use Milnor K-theory}
\label{Why use Milnor K-theory}
In the following, we explain why we use Milnor K-groups with support, i.e., certain eigenspaces of Thomason-Trobaugh K-groups,  not the entire Thomason-Trobaugh K-groups, to define cycles and Chow groups in Definition ~\ref{definition: Milnor K-theoretic Chow groups}. 

 In 2012 Fall, the author met a question on describing certain eigenspaces of K-groups and he E-mailed this question to Christophe Soul\'e for help.  Christophe Soul\'e suggested that if the author's question  is true, then it should be only true for Milnor K-theory and guided the author to read Theorem 5 in \cite{Soule}:  \\

  In our setting, $X$ is smooth projective over $k$, so the Gersten complex has the form of 
{\footnotesize
\begin{align*}
  0 \to K_{p}(k(X)) \to \dots \to \bigoplus_{x \in X^{(p-1)}}K_{1}(O_{X,x} \ \mathrm{on} \ x)
    \to \bigoplus_{x \in X^{(p)}}K_{0}(O_{X,x} \ \mathrm{on} \ x)  \to 0,
\end{align*}
}
which agrees with the Gersten complex by Quillen \cite{Quillen} because of D\'evissage:
{\footnotesize
\begin{align*}
  0 \to  K_{p}(k(X)) \to \dots \to \bigoplus_{x \in X^{(p-1)}}K_{1}(k(x)) 
    \to \bigoplus_{x \in X^{(p)}}K_{0}(k(x))  \to 0.
\end{align*}
}
 
For $x \in X^{(p)}$, Adams operations can decompose $K_{0}(O_{X,x} \ \mathrm{on} \ x)$ and $K_{0}(k(x))$ into direct sums of eigenspaces respectively. Moreover, Riemann-Roch without denominator, due to Soul\'e \cite{Soule}, says
\[
K^{(j)}_{0}(O_{X,x} \ \mathrm{on} \ x)_{\mathbb{Q}} = K^{(j-p)}_{0}(k(x))_{\mathbb{Q}}.
\]
  
For $j=p$,  
\[
K^{(p)}_{0}(O_{X,x} \ \mathrm{on} \ x)_{\mathbb{Q}} = K^{(0)}_{0}(k(x))_{\mathbb{Q}} = K_{0}(k(x))_{\mathbb{Q}},
\]
This forces to 
\[
K^{(j)}_{0}(O_{X,x} \ \mathrm{on} \ x)_{\mathbb{Q}} = 0, \ \mathrm{for} \  j \neq p.
\]
So only $K^{(p)}_{0}(O_{X,x} \ \mathrm{on} \ x)_{\mathbb{Q}}$ is needed to study $Z^{p}(X)_{\mathbb{Q}}$.

To give an example, for $X$ a smooth projective three-fold over a field $k$ of characteristic $0$,  a point $x \in X^{(3)}$ is defined by $(f, g,h)$ and a first order deformation of $x$ is given by $(f+ \varepsilon f_{1}, g+ \varepsilon g_{1},h+ \varepsilon h_{1})$. According to Theorem \ref{theorem: GilletSoule}, the Koszul complex $F_{\bullet}(f+ \varepsilon f_{1}, g+ \varepsilon g_{1},h+ \varepsilon h_{1})$ associated to $(f+ \varepsilon f_{1}, g+ \varepsilon g_{1},h+ \varepsilon h_{1})$ is of weight 3:
\[
F_{\bullet}(f+ \varepsilon f_{1}, g+ \varepsilon g_{1},h+ \varepsilon h_{1}) \in K^{(3)}_{0}(O_{X,x}[\varepsilon] \ \mathrm{on} \ x)_{\mathbb{Q}},
\]
and $F_{\bullet}(f+ \varepsilon f_{1}, g+ \varepsilon g_{1},h+ \varepsilon h_{1}) \notin K^{(2)}_{0}(O_{X,x}[\varepsilon] \ \mathrm{on} \ x)_{\mathbb{Q}}$.

So we ignore  $K^{(2)}_{0}(O_{X,x}[\varepsilon] \ \mathrm{on} \ x)_{\mathbb{Q}}(\cong H^{3}_{x}(O_{X}) \neq 0)$, and use only $K^{(3)}_{0}(O_{X,x}[\varepsilon] \ \mathrm{on} \ x)_{\mathbb{Q}}$
to define Milnor K-theoretic 3-cycles $Z^{M}_{3}(D^{\mathrm{Perf}}(X[\varepsilon]))$, which is the first order deformation of  $Z^{3}(X)_{\mathbb{Q}}$.

\section{\textbf{Higher order deformation-obstructions}}
\label{Obstruction issues}

Let $X$ be a smooth projective variety over a field $k$ of characteristic $0$. For each positive integer $j$, $X_{j}=X \times_{k}  \mathrm{Spec}(k[t]/ t^{j+1})$ is the $j$-th order infinitesimally trivial deformation of $X$.
For any integer $m$,  let $K^{M}_{m}(O_{X_{j},x_{j}} \ \mathrm{on} \ x_{j},t)$
denote the relative K-group, that is, the kernel of the natural projection
\[
 K^{M}_{m}(O_{X_{j},x_{j}} \ \mathrm{on} \ x_{j}) \xrightarrow{t =0} K^{M}_{m}(O_{X,x} \ \mathrm{on} \ x).
\]
Recall that we have proved the following isomorphisms in \cite{Y-2}:
\begin{theorem} [Corollary 3.11 in \cite{Y-2}] \label{theorem: CompRelatK}
Let $X$ be a smooth projective variety over a field $k$ of characteristic $0$ and 
let $x \in X^{(i)}$. Chern character induces the following isomorphisms between relative K-groups and local cohomology groups:
\[
 \begin{CD}
 K^{M}_{m}(O_{X_{j},x_{j}} \ \mathrm{on} \ x_{j},t) \cong H_{x}^{i}((\Omega^{m+i-1}_{X/\mathbb{Q}})^{\oplus j}).
 \end{CD}
 \]
\end{theorem}
So we have the following split exact sequence
\begin{equation}
 0 \to H_{x}^{i}((\Omega^{m+i-1}_{X/\mathbb{Q}})^{\oplus j}) \to K^{M}_{m}(O_{X_{j},x_{j}} \ \mathrm{on} \ x_{j}) \xrightarrow{t =0} K^{M}_{m}(O_{X,x} \ \mathrm{on} \ x) \to 0.
\end{equation}
Moreover, it is known that from the computation of Hochschild(cyclic) homology of truncated polynomials, $H_{x}^{i}((\Omega^{m+i-1}_{X/\mathbb{Q}})^{\oplus j})$ carries additional structure:
\begin{equation}
 H_{x}^{i}((\Omega^{m+i-1}_{X/\mathbb{Q}})^{\oplus j}) \cong t H_{x}^{i}(\Omega^{m+i-1}_{X/\mathbb{Q}}) \oplus \cdots \oplus t^{j}H_{x}^{i}(\Omega^{m+i-1}_{X/\mathbb{Q}}).
 \end{equation}
 
To simplify the notations, we use $A$ to denote $K^{M}_{m}(O_{X,x} \ \mathrm{on} \ x)$ and $B$ to denote $H_{x}^{i}(\Omega^{m+i-1}_{X/\mathbb{Q}})$, then we have
\begin{equation}
 K^{M}_{m}(O_{X_{j},x_{j}} \ \mathrm{on} \ x_{j}) \cong A \oplus tB \oplus \cdots t^{j}B. 
\end{equation}

The natural map
\[
f_{j}: X_{j} \to X_{j+1},
\]
induces $f^{\ast}_{j}:  K^{M}_{m}(O_{X_{j+1},x_{j+1}} \ \mathrm{on} \ x_{j+1})  \to K^{M}_{m}(O_{X_{j},x_{j}} \ \mathrm{on} \ x_{j})$. Moreover, 
there exists the following commutative diagram
\[
  \begin{CD}
      K^{M}_{m}(O_{X_{j+1},x_{j+1}} \ \mathrm{on} \ x_{j+1})    @>f_{j}^{*}>> K^{M}_{m}(O_{X_{j},x_{j}} \ \mathrm{on} \ x_{j})  \\
            @V \cong VV  @V \cong VV \\
         A \oplus tB \oplus \cdots t^{j}B \oplus t^{j+1}B @>t^{j+1}=0>> A \oplus tB \oplus \cdots t^{j}B,   
\end{CD}
\]  
and exists the following short exact sequence of abelian groups:
\[
0 \to B \to A \oplus tB \oplus \cdots t^{j}B \oplus t^{j+1}B \xrightarrow{t^{j+1}=0}
A \oplus tB \oplus \cdots t^{j}B \to 0.
\]

This shows that 
\begin{lemma} \label{lemma: keylemma}
For $X$ a smooth projective variety over a field $k$ of characteristic $0$, for each positive integer $j$ and $ x \in X^{(i)}$, there exists the following short exact sequence of abelian groups, where $m$ is any integer,
\[
0 \to H_{x}^{i}(\Omega^{m+i-1}_{X/\mathbb{Q}}) \to   K^{M}_{m}(O_{X_{j+1},x_{j+1}} \ \mathrm{on} \ x_{j+1})    \xrightarrow{f_{j}^{*}}
K^{M}_{m}(O_{X_{j},x_{j}} \ \mathrm{on} \ x_{j})   \to 0. 
\]
\end{lemma}

\subsection{Obstructions and negative K-groups}
\label{Obstructions and negative K-groups}
The natural map $f_{j}: X_{j} \to X_{j+1}$ induces the following commutative diagram:
\[
\begin{CD}
\bigoplus\limits_{x_{j+1} \in X_{j+1} ^{(p)}}K^{M}_{0}(O_{X_{j+1},x_{j+1}} \ \mathrm{on} \ x_{j+1})    @>f_{j}^{*}>>
\bigoplus\limits_{x_{j} \in X_{j} ^{(p)}}K^{M}_{0}(O_{X_{j},x_{j}} \ \mathrm{on} \ x_{j})  \\ 
@Vd_{1,X_{j+1}}^{p,-p}VV  @Vd_{1,X_{j}}^{p,-p}VV \\ 
 \bigoplus\limits_{x_{j+1} \in X_{j+1} ^{(p+1)}} K^{M}_{-1}(O_{X_{j+1},x_{j+1}} \ \mathrm{on} \ x_{j+1})   @>f_{j}^{*}>>
\bigoplus\limits_{x_{j} \in X_{j} ^{(p+1)}}K^{M}_{-1}(O_{X_{j},x_{j}} \ \mathrm{on} \ x_{j}),
\end{CD}
\]
so it further induces $f^{\ast}_{j}:  Z^{M}_{p}(D^{\mathrm{perf}}(X_{j+1})) \to  Z^{M}_{p}(D^{\mathrm{perf}}(X_{j}))$, recall that $Z^{M}_{p}(D^{\mathrm{perf}}(X_{j}))$ is defined as the kernel of $\mathrm{Ker}(d_{1,X_{j}}^{p,-p})$ in Definition ~\ref{definition: Milnor K-theoretic Chow groups}.

\begin{definition} \label{definition: deformation}
Given $\xi_{j} \in Z^{M}_{p}(D^{\mathrm{perf}}(X_{j}))$,  an element $\xi_{j+1} \in Z^{M}_{p}(D^{\mathrm{perf}}(X_{j+1}))$ is called a deformation of $\xi_{j}$,  if $f^{\ast}_{j}(\xi_{j+1}) = \xi_{j}$.
\end{definition}

$\xi_{j}$ and $\xi_{j+1}$ can be formally written as finite sums
\[
 \sum_{x_{j}}\lambda_{j}\cdot \overline{\{x_{j}\}}_{\mathrm{red}} \ \mathrm{and} \   \sum_{x_{j+1}}\lambda_{j+1}\cdot \overline{\{x_{j+1}\}}_{\mathrm{red}},
\] 
where $\sum\limits_{x_{j}}\lambda_{j} \in \mathrm{Ker}(d_{1,X_{j}}^{p,-p}) \subset \bigoplus\limits_{x_{j} \in X_{j} ^{(p)}}K^{M}_{0}(O_{X_{j},x_{j}} \ \mathrm{on} \ x_{j})$ and $\overline{\{x_{j}\}}_{\mathrm{red}}$ is the closed reduced scheme associated to $\overline{\{x_{j}\}}$. 

Since $\overline{\{x_{j}\}}_{\mathrm{red}} = \overline{\{x_{j+1}\}}_{\mathrm{red}}$, when we deform from  $\xi_{j}$ to $\xi_{j+1}$, we  deform the \textbf{coefficients}, i.e, we deform from  $\sum\limits_{x_{j}}\lambda_{j}$ to $\sum\limits_{x_{j+1}}\lambda_{j+1}$.

Since 
\[
f^{\ast}_{j}: \bigoplus\limits_{x_{j+1} \in X_{j+1} ^{(p)}}K^{M}_{0}(O_{X_{j+1},x_{j+1}} \ \mathrm{on} \ x_{j+1}) \to \bigoplus\limits_{x_{j} \in X_{j} ^{(p)}}K^{M}_{0}(O_{X_{j},x_{j}} \ \mathrm{on} \ x_{j})
\]
is surjective, see lemma ~\ref{lemma: keylemma}, given any $\xi_{j} \in Z^{M}_{p}(D^{\mathrm{perf}}(X_{j}))$, there exists 
\[
\xi_{j+1} \in \bigoplus\limits_{x_{j+1} \in X_{j+1} ^{(p)}}K^{M}_{0}(O_{X_{j+1},x_{j+1}} \ \mathrm{on} \ x_{j+1})
\]
 such that $f^{\ast}_{j}(\xi_{j+1}) = \xi_{j}$. We would like to know whether $\xi_{j+1} \in Z^{M}_{p}(D^{\mathrm{perf}}(X_{j+1}))$.

An easy diagram chasing shows $f^{\ast}_{j}d_{1,X_{j+1}}^{p,-p}(\xi_{j+1})=0$, so $d_{1,X_{j+1}}^{p,-p}(\xi_{j+1}) \in \mathrm{Ker}(f^{\ast}_{j})= \bigoplus\limits_{x \in X^{(p+1)}}H_{x}^{p+1}(\Omega_{X/\mathbb{Q}}^{p-1})$, see lemma ~\ref{lemma: keylemma}(take $m=-1$ and $i=p+1$).  If $d_{1,X_{j+1}}^{p,-p}(\xi_{j+1})  =0$, then we can lift $\xi_{j}$ to $\xi_{j+1}$.

\begin{definition}  \label{definition: obstructionspace}
The obstruction space for lifting elements in  $Z^{M}_{p}(D^{\mathrm{perf}}(X_{j}))$ to $Z^{M}_{p}(D^{\mathrm{perf}}(X_{j+1}))$ is defined to be $\bigoplus\limits_{x \in X^{(p+1)}}H_{x}^{p+1}(\Omega_{X/\mathbb{Q}}^{p-1}) $.
\end{definition}

\subsection{Obstruction issues-versus Hilbert scheme}
\label{Obstruction issues}
For each positive integer $j$, let $X_{j}$ denote the $j$-th trivial deformation of $X$.
Let $Y \subset X$ be a subvariety of codimension $p$. Obstruction issues asks whether it is possible to lift $Y$ to $Y_{j}$ successively, where $Y_{j} \subset X_{j} $ with suitable assumptions.

It is a common phenomenon that obstructions can occur in deformation, though the deformation of $X$ is trivial. It is well known that, considering $Y$ as an element of $\mathrm{Hilb}(X)$,  the tangent space $T_{Y}\mathrm{Hilb}(X)$ may be obstructed.

However, Green-Griffiths predicts that, considering $Y$ as an element of $Z^{p}(X)$, we can eliminate obstructions in their program \cite{GGtangentspace}:

\textbf{Obstruction issues}(page 187-190 in \cite{GGtangentspace}):
There are essentially four (not mutually exclusive) possibilities:

\begin{itemize}
\item (i) $TZ^{p}(X)$ may be obstructed. That is, there exists some $\tau \in TZ^{p}(X)$
such that, thinking of $\tau$ as a map 
\[
\mathrm{Spec}(k[\varepsilon]/(\varepsilon^2)) \to Z^{p}(X),
\]
this map cannot be lifted to a map 
\[
\mathrm{Spec}(k[\varepsilon]/(\varepsilon^{k+1})) \to Z^{p}(X)
\]
for some $k \geq 2$. \\
 
\item (ii) $TZ^{p}(X)$  is formally unobstructed. That is, for any $\tau \in TZ^{p}(X)$, $\tau$ may be lifted 
to a map 
\[
\mathrm{lim}(\mathrm{Spec}(k[\varepsilon]/(\varepsilon^{k+1}))) \to Z^{p}(X).
\] 

\item(iiii) $TZ^{p}(X)$  is formally unobstructed, but there exists $\tau \in TZ^{p}(X)$ which is not the tangent to a  geometric arc  in $Z^{p}(X)$. \\

\item (iv)  Every $\tau \in TZ^{p}(X)$ is the tangent to a geometric arc  in $Z^{p}(X)$. \\
 
\end{itemize}

For $p=1$,  this question was solved by TingFai Ng in his Ph.D thesis,
\begin{theorem} [Theorem 1.3.3 in \cite{Ng}] \label{theorem: Ng's theorem} Every $\tau \in TZ^{1}(X)$ is the tangent to a geometric arc in $Z^{1}(X)$.
\end{theorem} 

For $p \geq 2$, Green-Griffiths observes that 
\begin{proposition}[(10.11) on page 189 in \cite{GGtangentspace}]
For $p \geq 2$,  there exits $X$ and $\tau \in TZ^{p}(X)$ which is not the tangent to a geometric arc in $Z^{p}(X)$.
\end{proposition}
This means only possibilities (i)-(iii) can occur for $p \geq 2$. Green-Griffiths conjectures that 
\begin{conjecture}  [page 190 in \cite{GGtangentspace}] \label{question: obstruction}
(ii) and (iii) above are the only possibilities that actually occur for $p \geq 2$.
\end{conjecture}

Because of the Proposition 3.6 above, all we need to show is $TZ^{p}(X)$  is  formally unobstructed. The above question(ii) is expressed in a way, as if $Z^{p}(X)$ were a scheme.  In fact, we know $Z^{p}(X)$ can't be treated as a scheme. So we restate this conjecture as follows:  
\begin{conjecture}   \cite{GGtangentspace} \label{conjecture: GGquestion2}
 Let $X$ be a smooth projective variety over a field $k$ of characteristic $0$. For each positive integer $p$, $TZ^{p}(X)$  is formally unobstructed. That is, for any $\tau \in TZ^{p}(X)$, $\tau$ can be lifted 
to $\tau_{j} \in Z^{M}_{p}(D^{\mathrm{perf}}(X_{j})$ successively, where $j=1,2, \cdots$. 
\end{conjecture}

To get a feeling of  how to eliminate obstructions to deforming cycles, we firstly look at locally complete intersections.

For $X$ a smooth projective variety over a field $k$ of characteristic $0$ and $Y \subset X$ a subvariety, which is a locally complete intersection of codimension $p$. We assume that, on an open affine $U_{i} \subset X$,  
$Y \cap U_{i}$ is defined by a regular sequence $(f^{i}_{1}, \cdots, f^{i}_{p})$, where $f^{i}_{\ast} \in O_{X}(U_{i})$.  On another open affine $U_{j} \subset X$,  $Y \cap U_{j}$ is defined by a regular sequence $(f^{j}_{1}, \cdots, f^{j}_{p})$, where $f^{j}_{\ast} \in O_{X}(U_{j})$.

Let $Y'$ be a first order deformation of $Y$ in $X[\varepsilon]$, then 
$Y' \cap U_{i}$ is given by lifting $(f^{i}_{1}, \cdots, f^{i}_{p})$ to $(f^{i}_{1}+ \varepsilon g^{i}_{1}, \cdots, f^{i}_{p}+ \varepsilon g^{i}_{p})$ , where $g^{i}_{\ast} \in O_{X}(U_{i})$.  And $Y' \cap U_{j}$ is given by lifting $(f^{j}_{1}, \cdots, f^{j}_{p})$ to $(f^{j}_{1}+ \varepsilon g^{j}_{1}, \cdots, f^{j}_{p}+ \varepsilon g^{j}_{p})$ , where $g^{j}_{\ast} \in O_{X}(U_{j})$.

On the intersection $U_{ij}=U_{i} \cap U_{j}$, there exists two liftings which defines an element of $\alpha_{ij} \in \Gamma(U_{ij}, \mathcal{N}_{Y/X})$,
where $\mathcal{N}_{Y/X}$ is the normal sheaf. On the intersection $U_{ijk}=U_{i} \cap U_{j} \cap U_{k}$ of three open affine subschemes, there are three liftings which defines $\alpha_{ij}$, $\alpha_{jk}$, $\alpha_{ik}$.  One checks $(\alpha_{ij})$ is a \v{C}ech 1-cocyle, which is the obstruction of finding a global lifting $Y'$, see Theorem 6.2(page 47) of \cite{Hartshorne2} for details.

 Let $y \in Y$ be the generic point, then $y \in U_{i}$.
 One has $O_{X,y} = O_{U_{i},y}=O_{X}(U_{i})_{(f^{i}_{1}, \cdots, f^{i}_{p})}$, with maximal ideal $(f^{i}_{1}, \cdots, f^{i}_{p})$. So $Y$ is generically generated by $(f^{i}_{1}, \cdots, f^{i}_{p})$ and the Koszul complex $F_{\bullet}(f^{i}_{1}, \cdots, f^{i}_{p}) \in K^{(p)}_{0}(O_{X,y} \ \mathrm{on} \  y) \subset  Z^{M}_{p}(D^{\mathrm{perf}}(X))$.

 We have shown that, in Theorem \ref{theorem: Inkernel-LCI}, the Koszul complex $F_{\bullet}(f^{i}_{1}+ \varepsilon g^{i}_{1}, \cdots, f^{i}_{p}+ \varepsilon g^{i}_{p}) \in Z^{M}_{p}(D^{\mathrm{perf}}(X_{1}))$, which lifts $F_{\bullet}(f^{i}_{1}, \cdots, f^{i}_{p})$. So the obstructions of gluing(as a subscheme) $Y' \cap U_{i}$ and $Y' \cap U_{j}$ along the intersection $U_{ij}=U_{i} \cap U_{j}$ obvious vanishes(To lift K-theoretic cycles, we don't need to glue).
 By mimicking the proof of Theorem \ref{theorem: Inkernel-LCI}, we can show that the Koszul complex $F_{\bullet}(f^{i}_{1}+ \varepsilon g^{i}_{1}+ \varepsilon^{2}h^{i}_{1}, \cdots, f^{i}_{p}+ \varepsilon g^{i}_{p}+\varepsilon^{2}h^{i}_{p})$, where $\varepsilon^{2} \neq 0$, $\varepsilon^{3} = 0$, and $h^{i}_{\ast} \in O_{X}(U_{i})$, is a Milnor K-theoretic $p$-cycle and lifts $F_{\bullet}(f^{i}_{1}+ \varepsilon g^{i}_{1}, \cdots, f^{i}_{p}+ \varepsilon g^{i}_{p})$. Furthermore, we can lift $F_{\bullet}(f^{i}_{1}+ \varepsilon g^{i}_{1}+ \varepsilon^{2}h^{i}_{1}, \cdots, f^{i}_{p}+ \varepsilon g^{i}_{p}+\varepsilon^{2}h^{i}_{p})$ to higher order successively.  In summary, we have shown that
\begin{lemma} \label{lemma: LCI-ObsVanish}
For $X$ a smooth projective variety over a field $k$ of characteristic $0$ and $Y \subset X$ a subvariety, which is locally defined by a regular sequence $(f_{1}, \cdots, f_{p})$,  let $F_{\bullet}(f_{1}, \cdots, f_{p})$ denote the associated Koszul complex, which defines a K-theoretic cycle in $Z^{M}_{p}(D^{\mathrm{perf}}(X))$, 
we can lift  this K-theoretic cycle to higher order successively.
\end{lemma}

In general, $Y \subset X$ may not be a locally complete intersection. To eliminate the obstructions to lifting $Y$ to higher order, we need to use the following strategy which has been known to Green-Griffiths  \cite{GGtangentspace}(page 187-189) and Ng \cite{Ng} for the divisor case.
We should introduce another cycle $Z$ to help $Y$ to eliminate obstructions. As a cycle, 
\[
Y = (Y+Z)-Z,
\]
and the cycle $Z$ should satisfy that 
\begin{itemize}
\item (1) One can lift $(Y+Z)$ to higher order successively, i.e., $Z$ helps $Y$ to eliminate obstructions.

\item  (2) $Z$ doesn't introduce new obstructions. 
\end{itemize}

To illustrate the idea, we sketch an example of curves on a three-fold and refer the readers to \cite{Y-4} for details. For $X$ a nonsingular projective 3-fold over a field $k$ of characteristic $0$, let $Y \subset X$ be a curve with generic point $y$. 
For a point $x \in Y \subset X$ which is defined by $(f ,g, h)$, we assume $Y$
is generically defined by $(f, g)$. The Koszul complex $F_{\bullet}(f, g)$ is a K-theoretic 2-cycle:
\[
F_{\bullet}(f, g)  \in K_{0}(O_{X,y} \ \mathrm{on} \ y)  \subset Z^{M}_{2}(D^{\mathrm{perf}}(X)).
\]

For a first order deformation $Y'$ which is generically given by $(f + \varepsilon \dfrac{1}{h}, g)$, the Koszul complex $F_{\bullet}(f + \varepsilon \dfrac{1}{h}, g)$ associated to $(f + \varepsilon \dfrac{1}{h}, g)$ is in $K_{0}(O_{X,y}[\varepsilon] \ \mathrm{on} \ y) $, but we can show it is not in $ Z^{M}_{2}(D^{\mathrm{perf}}(X[\varepsilon]))$, see Example 4.4 in \cite{Y-3}.
 So $F_{\bullet}(f + \varepsilon \dfrac{1}{h}, g)$ is not a first order deformation of  $F_{\bullet}(f , g)$. To modify this, we consider the curve $Z$ on $X$ which is generically defined by $(h, g)$. As a cycle, 
\[
Y = (Y+Z)-Z.
\]
As a K-theoretic cycle, 
\[
F_{\bullet}(f , g)= (F_{\bullet}(f , g)+F_{\bullet}(h , g))-F_{\bullet}(h , g).
\]
To lift $F_{\bullet}(f , g)$ is equivalent to lifting $(F_{\bullet}(f , g)+F_{\bullet}(h , g))$ and $F_{\bullet}(h , g)$ respectively.  We can show that $(F_{\bullet}(f + \varepsilon \dfrac{1}{h}, g)+F_{\bullet}(h + \varepsilon \dfrac{1}{f} , g))$ is in $Z^{M}_{2}(D^{\mathrm{perf}}(X[\varepsilon]))$.
And it is a first order deformation of $(F_{\bullet}(f , g)+F_{\bullet}(h , g))$, and can be lifted to higher order successively. On the other hand, $F_{\bullet}(h , g)$ is always a first order deformation of itself, which means we fix $F_{\bullet}(h , g)$ so it doesn't introduce new obstructions. Consequently, 
\[
(F_{\bullet}(f + \varepsilon \dfrac{1}{h}, g)+F_{\bullet}(h + \varepsilon \dfrac{1}{f} , g)) - F_{\bullet}(h , g)
\]
is a first order deformation of $F_{\bullet}(f , g)$, and can be lifted to higher order successively.

However, as pointed out in Remark \ref{remark: notallKoszul}, in general, we don't know whether the Milnor K-theoretic cycles $Z^{M}_{p}(D^{\mathrm{perf}}(X_{j}))$ are generated by these Koszul complexes or not. So, to answer Green-Griffiths' Conjecture \ref{conjecture: GGquestion2}, we have to give a formal argument which relys on the following theorem:

\newpage

\begin{theorem} [\cite{DHY}, Theorem 3.14 in \cite{Y-2}]  \label{theorem: theoremFor obstruction issue}
For $X$ a smooth projective variety over a field k of characteristic 0, $X_{j} = X \times_{k} \mathrm{Spec}(k[t]/(t^{j+1}))$,  where $j$ is any positive integer. For each integer $p\geqslant 1$, there exists the following commutative diagram in which the Zariski sheafification of each column is a flasque resolution of $ (\Omega_{X/ \mathbb{Q}}^{p-1})^{\oplus j}$,  $K^{M}_{p}(O_{X_{j}})$ and $K^{M}_{p}(O_{X})$
respectively. The left arrows are induced by Chern characters from K-theory to negative cyclic homology and the right ones are the natural maps sending $\varepsilon$ to $0$:
{\scriptsize
\[
  \begin{CD}
     0 @. 0 @. 0\\
      @VVV @VVV @VVV\\
     (\Omega_{k(X)/ \mathbb{Q}}^{p-1})^{\oplus j} @<<< K^{M}_{p}(k(X_{j})) @>>> K^{M}_{p}(k(X))  \\
     @VVV @VVV @VVV\\
      \bigoplus\limits_{x \in X^{(1)}}H_{x}^{1}((\Omega_{X/\mathbb{Q}}^{p-1})^{\oplus j}) @<<< \bigoplus\limits_{x_{j} \in X_{j} ^{(1)}}K^{M}_{p-1}(O_{X_{j},x_{j}} \ \mathrm{on} \ x_{j}) @>>>  \bigoplus\limits_{x \in X ^{(1)}}K^{M}_{p-1}(O_{X,x} \ \mathrm{on} \ x) \\
     @VVV @VVV @VVV\\
      \dots @<<< \dots @>>> \dots \\ 
      @VVV @VVV @VVV\\
     \bigoplus\limits_{x \in X^{(p-1)}}H_{x}^{p-1}((\Omega_{X/\mathbb{Q}}^{p-1})^{\oplus j}) @<<< \bigoplus\limits_{x_{j} \in X_{j} ^{(p-1)}}K^{M}_{1}(O_{X_{j},x_{j}} \ \mathrm{on} \ x_{j})
      @>>>  \bigoplus\limits_{x \in X^{(p-1)}}K^{M}_{1}(O_{X,x} \ \mathrm{on} \ x) \\
     @V \partial_{1,j}^{p-1,-p}VV @Vd_{1,X_{j}}^{p-1,-p}VV @Vd_{1, X}^{p-1,-p}VV\\
     \bigoplus\limits_{x \in X^{(p)}}H_{x}^{p}((\Omega_{X/\mathbb{Q}}^{p-1})^{\oplus j}) @<<< \bigoplus\limits_{x_{j} \in X_{j} ^{(p)}}K^{M}_{0}(O_{X_{j},x_{j}} \ \mathrm{on} \ x_{j})
      @>>> \bigoplus\limits_{x \in X^{(p)}}K^{M}_{0}(O_{X,x} \ \mathrm{on} \ x) \\
     @V \partial_{1,j}^{p,-p}VV @Vd_{1,X_{j}}^{p,-p}VV @Vd_{1,X}^{p,-p}VV\\
     \bigoplus\limits_{x \in X^{(p+1)}}H_{x}^{p+1}((\Omega_{X/\mathbb{Q}}^{p-1})^{\oplus j}) @< \cong<< \bigoplus\limits_{x_{j} \in X_{j} ^{(p+1)}}K^{M}_{-1}(O_{X_{j},x_{j}} \ \mathrm{on} \ x_{j})
      @>>> \bigoplus\limits_{x \in X^{(p+1)}}K^{M}_{-1}(O_{X,x} \ \mathrm{on} \ x) =0 \\
     @VVV @VVV @VVV\\
     \cdots @<<< \cdots @>>> \cdots \\ 
     @VVV @VVV @VVV\\
     \bigoplus\limits_{x \in X^{(d)}}H_{x}^{d}((\Omega_{X/\mathbb{Q}}^{p-1})^{\oplus j}) @<<<  \bigoplus\limits_{x_{j} \in X_{j}^{(d)}}K^{M}_{q-d}(O_{X_{j},x_{j}} \ \mathrm{on} \ x_{j})
   @>>> \bigoplus\limits_{x \in X^{(d)}}K^{M}_{q-d}(O_{X,x} \ \mathrm{on} \ x)  \\
     @VVV @VVV @VVV\\
      0 @. 0 @. 0.
  \end{CD}
\]
}
\end{theorem}

Using this theorem,  we answer Green-Griffiths' Conjecture ~\ref{conjecture: GGquestion2} affirmatively:
\begin{theorem} \label{theorem: theorem obstruction}
The Conjecture \ ~\ref{conjecture: GGquestion2} is true, that is, $TZ^{p}(X)$  is  formally unobstructed.
\end{theorem}

\begin{proof}
For any positive integer $j$ and given any $\xi_{j} \in Z^{M}_{p}(D^{\mathrm{perf}}(X_{j}))(:=\mathrm{Ker}(d_{1,X_{j}}^{p,-p}))$, we need to show $\xi_{j}$ can be lifted to an element of $Z^{M}_{p}(D^{\mathrm{perf}}(X_{j+1}))(:=\mathrm{Ker}(d_{1,X_{j+1}}^{p,-p}))$. 
There exists the commutative diagram(part of the diagram in Theorem ~\ref{theorem: theoremFor obstruction issue}),
\[
  \begin{CD}
     \bigoplus\limits_{x \in X^{(p)}}H_{x}^{p}((\Omega_{X/\mathbb{Q}}^{p-1})^{\oplus j}) @<\mathrm{Ch}<< \bigoplus\limits_{x_{j} \in X_{j} ^{(p)}}K^{M}_{0}(O_{X_{j},x_{j}} \ \mathrm{on} \ x_{j}) \\ 
     @V \partial_{1,j}^{p,-p}VV @Vd_{1,X_{j}}^{p,-p}VV \\
     \bigoplus\limits_{x \in X^{(p+1)}}H_{x}^{p+1}((\Omega_{X/\mathbb{Q}}^{p-1})^{\oplus j}) @<\mathrm{Ch}<\cong< \bigoplus\limits_{x_{j} \in X_{j} ^{(p+1)}}K^{M}_{-1}(O_{X_{j},x_{j}} \ \mathrm{on} \ x_{j}), \\
  \end{CD}
\]
where the maps Ch are induced by Chern characters from K-theory to negative cyclic homology. It is obvious that $\mathrm{Ch}(\xi_{j}) \in \mathrm{Ker}(\partial_{1,j}^{p,-p})$.

There exists a similar commutative diagram for j+1:
\[
  \begin{CD}
     \bigoplus\limits_{x \in X^{(p)}}H_{x}^{p}((\Omega_{X/\mathbb{Q}}^{p-1})^{\oplus j+1}) @<\mathrm{Ch}<< \bigoplus\limits_{x_{j+1} \in X_{j+1} ^{(p)}}K^{M}_{0}(O_{X_{j+1},x_{j+1}} \ \mathrm{on} \ x_{j+1}) \\
     @V \partial_{1, j+1}^{p,-p}VV @Vd_{1,X_{j+1}}^{p,-p}VV \\
     \bigoplus\limits_{x \in X^{(p+1)}}H_{x}^{p+1}((\Omega_{X/\mathbb{Q}}^{p-1})^{\oplus j+1}) @<\mathrm{Ch}<\cong< \bigoplus\limits_{x_{j+1} \in X_{j+1} ^{(p+1)}}K^{M}_{-1}(O_{X_{j+1},x_{j+1}} \ \mathrm{on} \ x_{j+1}). \\
  \end{CD}
\]

As explained on page 20(isomorphism (3.2)),  $\bigoplus\limits_{x \in X^{(p)}}H_{x}^{p}((\Omega_{X/\mathbb{Q}}^{p-1})^{\oplus j})$ carries additional structure:
\[
\bigoplus\limits_{x \in X^{(p)}}H_{x}^{p}((\Omega_{X/\mathbb{Q}}^{p-1})^{\oplus j}) \cong t \bigoplus\limits_{x \in X^{(p)}}H_{x}^{p}((\Omega_{X/\mathbb{Q}}^{p-1})) \oplus \cdots \oplus t^{j} \bigoplus\limits_{x \in X^{(p)}}H_{x}^{p}((\Omega_{X/\mathbb{Q}}^{p-1})).
\]
And the differential 
\[
\partial_{1,j}^{p,-p}: \bigoplus\limits_{x \in X^{(p)}}H_{x}^{p}((\Omega_{X/\mathbb{Q}}^{p-1})^{\oplus j}) \to \bigoplus\limits_{x \in X^{(p+1)}}H_{x}^{p+1}((\Omega_{X/\mathbb{Q}}^{p-1})^{\oplus j}) 
\]
is $t \partial_{1}^{p,-p} \oplus \cdots \oplus t^{j} \partial_{1}^{p,-p}$:
\[
\begin{CD}
\bigoplus\limits_{x \in X^{(p)}}H_{x}^{p}((\Omega_{X/\mathbb{Q}}^{p-1})^{\oplus j})  @>\cong>>  t \bigoplus\limits_{x \in X^{(p)}}H_{x}^{p}((\Omega_{X/\mathbb{Q}}^{p-1})) \oplus \cdots \oplus t^{j} \bigoplus\limits_{x \in X^{(p)}}H_{x}^{p}((\Omega_{X/\mathbb{Q}}^{p-1})) \\
@V \partial_{1}^{p,-p}VV @Vt \partial_{1}^{p,-p} \oplus \cdots \oplus t^{j} \partial_{1}^{p,-p}VV \\
\bigoplus\limits_{x \in X^{(p+1)}}H_{x}^{p+1}((\Omega_{X/\mathbb{Q}}^{p-1})^{\oplus j})  @>\cong>>  t \bigoplus\limits_{x \in X^{(p+1)}}H_{x}^{p+1}((\Omega_{X/\mathbb{Q}}^{p-1})) \oplus \cdots \oplus t^{j} \bigoplus\limits_{x \in X^{(p+1)}}H_{x}^{p+1}((\Omega_{X/\mathbb{Q}}^{p-1})),
\end{CD}
\]
where $\partial_{1}^{p,-p}: \bigoplus\limits_{x \in X^{(p)}}H_{x}^{p}((\Omega_{X/\mathbb{Q}}^{p-1})) \to \bigoplus\limits_{x \in X^{(p+1)}}H_{x}^{p+1}((\Omega_{X/\mathbb{Q}}^{p-1})) $.

Under these isomorphisms,  $\mathrm{Ch}(\xi_{j})$ can be written as $ta_{1}+ \cdots + t^{j}a_{j}$, where each $a_{i} \in \bigoplus\limits_{x \in X^{(p)}}H_{x}^{p}((\Omega_{X/\mathbb{Q}}^{p-1}))$ and $\partial_{1}^{p,-p}(a_{i})=0$.
There exists a similar isomorphism for $j+1$:
\[
\bigoplus\limits_{x \in X^{(p)}}H_{x}^{p}((\Omega_{X/\mathbb{Q}}^{p-1})^{\oplus j+1}) \cong t \bigoplus\limits_{x \in X^{(p)}}H_{x}^{p}((\Omega_{X/\mathbb{Q}}^{p-1})) \oplus \cdots \oplus t^{j+1} \bigoplus\limits_{x \in X^{(p)}}H_{x}^{p}((\Omega_{X/\mathbb{Q}}^{p-1})).
\]
And the differential 
\[
\partial_{1, j+1}^{p,-p}: \bigoplus\limits_{x \in X^{(p)}}H_{x}^{p}((\Omega_{X/\mathbb{Q}}^{p-1})^{\oplus j+1}) \to \bigoplus\limits_{x \in X^{(p+1)}}H_{x}^{p+1}((\Omega_{X/\mathbb{Q}}^{p-1})^{\oplus j+1}) 
\]
is $t \partial_{1}^{p,-p} \oplus \cdots \oplus t^{j+1} \partial_{1}^{p,-p}$.

We can always lift $ta_{1}+ \cdots + t^{j}a_{j}$ to $\eta_{j+1} := ta_{1}+ \cdots + t^{j}a_{j} + t^{j+1}a_{j+1}$(note $t^{j+1} \neq 0$ here), where $a_{j+1} \in \mathrm{Ker}(\partial_{1}^{p,-p})$. So $\eta_{j+1}  \in \mathrm{Ker}(\partial_{1,j+1}^{p,-p})$. Hence, we can always lift 
$\mathrm{Ch}(\xi_{j})$ to $\eta_{j+1} \in \mathrm{Ker}(\partial_{1,j+1}^{p,-p})$.

Since the map 
\[
\mathrm{Ch}: \bigoplus\limits_{x_{j+1} \in X_{j+1} ^{(p)}}K^{M}_{0}(O_{X_{j+1},x_{j+1}} \ \mathrm{on} \ x_{j+1}) \to \bigoplus\limits_{x \in X^{(p)}}H_{x}^{p}((\Omega_{X/\mathbb{Q}}^{p-1})^{\oplus j+1})
\]
is surjective, there exists $\xi_{j+1} \in \bigoplus\limits_{x_{j+1} \in X_{j+1} ^{(p)}}K^{M}_{0}(O_{X_{j+1},x_{j+1}} \ \mathrm{on} \ x_{j+1})$ such that $\mathrm{Ch}(\xi_{j+1})=\eta_{j+1} $.

By the naturality of Chern character, there exists the following commutative diagram:
\[
  \begin{CD}
     \bigoplus\limits_{x \in X^{(p)}}H_{x}^{p}((\Omega_{X/\mathbb{Q}}^{p-1})^{\oplus j+1}) @<\mathrm{Ch}<< \bigoplus\limits_{x_{j+1} \in X_{j+1} ^{(p)}}K^{M}_{0}(O_{X_{j+1},x_{j+1}} \ \mathrm{on} \ x_{j+1}) \\
     @Vt^{j+1}=0VV @Vt^{j+1}=0VV \\
      \bigoplus\limits_{x \in X^{(p)}}H_{x}^{p}((\Omega_{X/\mathbb{Q}}^{q-1})^{\oplus j}) @<\mathrm{Ch}<< \bigoplus\limits_{x_{j} \in X_{j} ^{(p)}}K^{M}_{0}(O_{X_{j},x_{j}} \ \mathrm{on} \ x_{j}).
  \end{CD}
\]

So we have the following commutative diagram:
\[
  \begin{CD}
     \eta_{j+1} = \mathrm{Ch}(\xi_{j+1})@<\mathrm{Ch}<<  \xi_{j+1} \\ 
     @Vt^{j+1}=0VV @Vt^{j+1}=0VV \\
      \eta_{j+1}|_{t^{j+1}=0} @<\mathrm{Ch}<<  \xi_{j+1}|_{t^{j+1=0}}.
  \end{CD}
\]
This says $\eta_{j+1}|_{t^{j+1}=0} = \mathrm{Ch}(\xi_{j+1}|_{t^{j+1}=0})$. On the other hand, since $\eta_{j+1}$ lifts $\mathrm{Ch}(\xi_{j})$, $\eta_{j+1}|_{t^{j+1}=0} = \mathrm{Ch}(\xi_{j})$. Hence, $\xi_{j+1}|_{t^{j+1}=0} - \xi_{j}$ is in the kernel of the map
\[
 \mathrm{Ch}:  \bigoplus\limits_{x_{j} \in X_{j} ^{(p)}}K^{M}_{0}(O_{X_{j},x_{j}} \ \mathrm{on} \ x_{j}) \to \bigoplus\limits_{x \in X^{(p)}}H_{x}^{p}((\Omega_{X/\mathbb{Q}}^{p-1})^{\oplus j}),
\]
which is $\bigoplus\limits_{x \in X^{(p)}}K^{M}_{0}(O_{X,x} \ \mathrm{on} \ x)$.
In other words, $\xi_{j+1}|_{t^{j+1}=0} = \xi_{j} + W$, for some $W \in \bigoplus\limits_{x \in X^{(p)}}K^{M}_{0}(O_{X,x} \ \mathrm{on} \ x)$.

As a \textbf{cycle}, $\xi_{j}$ can be written as a formal sum
\begin{equation}
\xi_{j} = (\xi_{j} + W) - W.
\end{equation}
Here, since $\bigoplus\limits_{x \in X^{(p)}}K^{M}_{0}(O_{X,x} \ \mathrm{on} \ x)$ is a direct summand of $\bigoplus\limits_{x_{j} \in X_{j} ^{(p)}}K^{M}_{0}(O_{X_{j},x_{j}} \ \mathrm{on} \ x_{j})$, both $W$ and $\xi_{j} + W$ are in $\bigoplus\limits_{x_{j} \in X_{j} ^{(p)}}K^{M}_{0}(O_{X_{j},x_{j}} \ \mathrm{on} \ x_{j})$.

Similarly, since $\bigoplus\limits_{x \in X^{(p)}}K^{M}_{0}(O_{X,x} \ \mathrm{on} \ x)$ is also a direct summand of $\bigoplus\limits_{x_{j+1} \in X_{j+1} ^{(p)}}K^{M}_{0}(O_{X_{j+1},x_{j+1}} \ \mathrm{on} \ x_{j+1})$, $W \in \bigoplus\limits_{x_{j+1} \in X_{j+1} ^{(p)}}K^{M}_{0}(O_{X_{j+1},x_{j+1}} \ \mathrm{on} \ x_{j+1})$ and 
the cycle $\xi_{j+1} -W \in \bigoplus\limits_{x_{j+1} \in X_{j+1} ^{(p)}}K^{M}_{0}(O_{X_{j+1},x_{j+1}} \ \mathrm{on} \ x_{j+1})$ satisfies
\[
(\xi_{j+1} -W )|_{t^{j+1}=0} = \xi_{j+1}|_{t^{j+1}=0} - W = \xi_{j} + W -W = \xi_{j}.
\]

Moreover, $\mathrm{Ch}(\xi_{j+1}-W)= \mathrm{Ch}(\xi_{j+1}) = \eta_{j+1} \in \mathrm{Ker}(\partial_{1,j+1}^{p,-p})$, hence, $\xi_{j+1} -W \in Z^{M}_{p}(D^{\mathrm{perf}}(X_{j+1}))(:=\mathrm{Ker}(d_{1,X_{j+1}}^{p,-p}))$ because of the commutative diagram
\[
  \begin{CD}
     \bigoplus\limits_{x \in X^{(p)}}H_{x}^{p}((\Omega_{X/\mathbb{Q}}^{p-1})^{\oplus j+1}) @<\mathrm{Ch}<< \bigoplus\limits_{x_{j+1} \in X_{j+1} ^{(p)}}K^{M}_{0}(O_{X_{j+1},x_{j+1}} \ \mathrm{on} \ x_{j+1}) \\ 
     @V \partial_{1,j+1}^{p,-p}VV @Vd_{1,X_{j+1}}^{p,-p}VV \\
     \bigoplus\limits_{x \in X^{(p+1)}}H_{x}^{p+1}((\Omega_{X/\mathbb{Q}}^{p-1})^{\oplus j+1}) @<\mathrm{Ch}<\cong< \bigoplus\limits_{x_{j+1} \in X_{j+1} ^{(p+1)}}K^{M}_{-1}(O_{X_{j+1},x_{j+1}} \ \mathrm{on} \ x_{j+1}). \\
  \end{CD}
\]

In conclusion, $\xi_{j+1} -W \in Z^{M}_{p}(D^{\mathrm{perf}}(X_{j+1}))(:=\mathrm{Ker}(d_{1,X_{j+1}}^{p,-p}))$ can lift $\xi_{j}$.
\end{proof}

In Section 4 of \cite{GGdeformation}, Green-Griffiths conjectures that 
\begin{conjecture} [(4.7) on page 506 \cite{GGdeformation}] \label{question: RatEqi Unobst}
Let $X$ be a smooth projective variety over a field $k$ of characteristic $0$, for each positive integer $p$, $TZ^{p}_{rat}(X)$ is formally unobstructed.
\end{conjecture}

For any positive integer $j$ and given any $\eta_{j} \in Z^{M}_{p,rat}(D^{\mathrm{perf}}(X_{j}))(:= \mathrm{Im}(d_{1, X_{j}}^{p-1,-p}))$, we want to know whether  $\eta_{j}$ can be lifted to  $\eta_{j+1} \in Z^{M}_{p,rat}(D^{\mathrm{perf}}(X_{j+1}))$. 

By definition, $\eta_{j} = d_{1, X_{j}}^{p-1,-p}(\xi_{j})$, for some $\xi_{j} \in \bigoplus\limits_{x_{j} \in X_{j} ^{(p-1)}}K^{M}_{1}(O_{X_{j},x_{j}} \ \mathrm{on} \ x_{j})$.
Since 
\[
f^{\ast}_{j}: \bigoplus_{x_{j+1} \in X_{j+1} ^{(p-1)}}K^{M}_{1}(O_{X_{j+1},x_{j+1}} \ \mathrm{on} \ x_{j+1}) \to \bigoplus_{x_{j} \in X_{j} ^{(p-1)}}K^{M}_{1}(O_{X_{j},x_{j}} \ \mathrm{on} \ x_{j})
\]
is surjective, see lemma ~\ref{lemma: keylemma}, we can always lift $\xi_{j}$ to $\xi_{j+1} \in \bigoplus\limits_{x_{j+1} \in X_{j+1} ^{(p-1)}}K^{M}_{1}(O_{X_{j+1},x_{j+1}} \ \mathrm{on} \ x_{j+1})$. Then $d_{1, X_{j+1}}^{p-1,-p}(\xi_{j+1})$ lifts $\eta_{j}$ because of the following commutative diagram:
\[
\begin{CD}
\bigoplus\limits_{x_{j+1} \in X_{j+1} ^{(p-1)}}K^{M}_{1}(O_{X_{j+1},x_{j+1}} \ \mathrm{on} \ x_{j+1})  @>f^{\ast}_{j}>> \bigoplus\limits_{x_{j} \in X_{j} ^{(p-1)}}K^{M}_{1}(O_{X_{j},x_{j}} \ \mathrm{on} \ x_{j}) \\
@Vd_{1, X_{j+1}}^{p-1,-p}VV  @Vd_{1, X_{j}}^{p-1,-p}VV \\ 
\bigoplus\limits_{x_{j+1} \in X_{j+1} ^{(p-1)}}K^{M}_{0}(O_{X_{j+1},x_{j+1}} \ \mathrm{on} \ x_{j+1})  @>f^{\ast}_{j}>> \bigoplus\limits_{x_{j} \in X_{j} ^{(p-1)}}K^{M}_{0}(O_{X_{j},x_{j}} \ \mathrm{on} \ x_{j}).
\end{CD}
\]

This proves the deformation from $Z^{M}_{p,rat}(D^{\mathrm{perf}}(X_{j}))$ to $Z^{M}_{p,rat}(D^{\mathrm{perf}}(X_{j+1}))$ is unobstructed. So we have
\begin{theorem} \label{theorem: theorem obstruction rat}
The conjecture \ref{question: RatEqi Unobst} is true, i.e., $TZ^{p}_{rat}(X)$  is unobstructed.
\end{theorem}

Recall that in Definition \ref{definition: Milnor K-theoretic Chow groups}, the $p$-th Milnor K-theoretic Chow group is defined to be:
\[
  CH^{M}_{p}(D^{\mathrm{perf}}(X_{j})) := \dfrac{Z^{M}_{p}(D^{\mathrm{Perf}}(X_{j}))}{Z^{M}_{p,rat}(D^{\mathrm{Perf}}(X_{j}))}.
\]

The proof of Theorem \ref{theorem: theorem obstruction} says , for any given $[\xi_{j}] \in CH^{M}_{p}(D^{\mathrm{perf}}(X_{j}))$, we can lift it to $[\xi_{j+1}] \in CH^{M}_{p}(D^{\mathrm{perf}}(X_{j+1}))$.

Recall that we have shown that, in \cite{Y-2}, $CH^{M}_{p}(D^{\mathrm{perf}}(X_{j}))$ satisfies Soul\'e's variant of Bloch-Quillen identification:
\[
CH^{M}_{p}(D^{\mathrm{perf}}(X_{j}))= H^{p}(X, K^{M}_{p}(O_{X_{j}}))_{\mathbb{Q}},
\]
where $K^{M}_{p}(O_{X_{j}})$ is the 
Milnor K-theory sheaf associated to the presheaf
\[
  U \to K^{M}_{p}(O_{X_{j}}(U)).
\]
So we have proved the following fact, which is already known to Green-Griffiths and can be deduced from Proposition 2.6 of  \cite{GGdeformation}(recalled below),
\begin{corollary} \cite{GGdeformation} \label{corollary: GG's corollary}
For each positive integer $j$, $X_{j}=X \times_{k} \mathrm{Spec}(k[t]/(t^{j+1}))$, for any given $[\xi_{j}] \in H^{p}(X, K^{M}_{p}(O_{X_{j}}))_{\mathbb{Q}}$, we can lift it to $[\xi_{j+1}] \in 
H^{p}(X, K^{M}_{p}(O_{X_{j+1}}))_{\mathbb{Q}}$.
\end{corollary}

We briefly explain how to prove this Corollary by Green-Griffiths \cite{GGdeformation}. As a trivial version of (2.1) of \cite{GGdeformation}(page 498) or (2.8) of Proposition 2.3 of \cite{BEK}, there exists the following short exact sequence of sheaves
\[
0 \to \Omega^{p-1}_{X/ \mathbb{Q}}  \to K^{M}_{p}(O_{X_{j+1}}) \to K^{M}_{p}(O_{X_{j}}) \to 0.
\]
The associated long exact sequence is of the form
\begin{equation}
\cdots \to H^{p}(X, K^{M}_{p}(O_{X_{j+1}}))_{\mathbb{Q}} \to H^{p}(X, K^{M}_{p}(O_{X_{j}}))_{\mathbb{Q}} \xrightarrow{\delta} H^{p+1}(X, \Omega^{p-1}_{X/ \mathbb{Q}}) \to \cdots.
\end{equation}

The arithmetic cycle mapping 
\[
\eta : H^{p}(X, K^{M}_{p}(O_{X_{j}})) \to H^{p}(X, \Omega^{p}_{X_{j}/ \mathbb{Q}})
\]
is induced by the dlog map
\[
K^{M}_{p}(O_{X_{j}}) \to \Omega^{p}_{X_{j}/ \mathbb{Q}}
\]
\[
\{ r_{1}, \cdots, r_{p} \} \to \dfrac{dr_{1}}{r_{1}} \wedge \cdots \wedge \dfrac{dr_{p}}{r_{p}},
\]
where $d=d_{\mathbb{Q}}$.

Let $\theta_{j}$ denote the $j$-th Kodaira-Spencer class, see Section 3.1(page 492) of \cite{GGdeformation} for the definition.
\begin{lemma} [Proposition 2.6 of \cite{GGdeformation}(page 502)]
The coboundary map $\delta$ in the above long exact sequence(3.5) is given by
\[
\delta(\xi_{j})=\theta_{j} \rfloor \eta(\xi_{j}).
\]
In other words, the obstruction to lifting $\xi_{j} \in H^{p}(X, K^{M}_{p}(O_{X_{j}}))$ to 
$H^{p}(X, K^{M}_{p}(O_{X_{j+1}}))$ is given by  
\[
\delta(\xi_{j})=\theta_{j} \rfloor \eta(\xi_{j}),
\]
where $\eta(\xi_{j})$ is the arithmetic cycle class of $\xi_{j}$. 
\end{lemma}

In our situation, $X_{j}=X \times \mathrm{Spec}(k[t]/(t^{j+1}))$, the Kodaira-Spencer class $\theta_{j}=0$(see page 492 of \cite{GGdeformation}), so the coboundary map $\delta = 0$:
\[
\cdots \to H^{p}(X, K^{M}_{p}(O_{X_{j+1}}))_{\mathbb{Q}} \to H^{p}(X, K^{M}_{p}(O_{X_{j}}))_{\mathbb{Q}}  \xrightarrow{\delta =0} H^{p+1}(X, \Omega^{p-1}_{X/ \mathbb{Q}})
\]
This proves Corollary \ref{corollary: GG's corollary} above.

\textbf{Acknowledgements}.
The author must record his deep gratitude to Mark Green and Phillip Griffiths for enlightening discussions and for their interest in this work. He is very grateful to Spencer Bloch\cite{Bloch1} and to Christophe Soul\'e(see Section 2.3) for discussions.

The author thanks several professors for their support: Ben Dribus, H\'el\`ene Esnault, Jerome Hoffman, Marc Levine, Kefeng Liu, Jan Stienstra, Hongwei Xu, Chao Zhang and anonymous comments on previous versions.

\end{document}